\documentclass[a4paper,11pt]{amsart}
\usepackage{amsmath,amsthm,amssymb,mathtools,thmtools,thm-restate}
\usepackage[mathscr]{eucal}
\usepackage{cite}
\usepackage{upgreek}
\usepackage[bookmarks=false]{hyperref}

\usepackage{enumerate}
\usepackage{mathrsfs}
\usepackage{blindtext}
\usepackage{scrextend}
\usepackage{enumitem}
\addtokomafont{labelinglabel}{\sffamily}
\usepackage{color}
\usepackage{nicefrac}
\usepackage[at]{easylist}
\usepackage{bbm}
\newcommand{\arXiv}[1]{\href{http://arxiv.org/abs/#1}{arXiv:#1}}

\numberwithin{equation}{section}

\theoremstyle{plain}
\newtheorem{mthm}{Theorem}

\newtheorem{theorem}{Theorem}[section]

\newtheorem{lemma}[theorem]{Lemma}
\newtheorem{proposition}[theorem]{Proposition}
\newtheorem{corollary}[theorem]{Corollary}
\theoremstyle{definition}
\newtheorem{definition}[theorem]{Definition}
\newtheorem*{definition*}{Definition}

\newtheorem{remark}[theorem]{Remark}

\newcommand{\C}{\mathbb{C}}

\newcommand{\bF}{\mathbb{F}}

\newcommand{\bM}{\mathbb{M}}
\newcommand{\N}{\mathbb{N}}\newcommand{\cN}{\mathcal{N}}

\newcommand{\cR}{\mathcal{R}}
\newcommand{\cS}{\mathcal{S}}
\newcommand{\bT}{\mathbb{T}}\newcommand{\cT}{\mathcal{T}}
\newcommand{\cU}{\mathcal{U}}

\newcommand{\Z}{\mathbb{Z}}\newcommand{\cZ}{\mathcal{Z}}

\newcommand{\al}{\alpha}
\newcommand{\eps}{\varepsilon}

\newcommand{\norm}[1]{\left\|#1\right\|}
\newcommand{\abs}[1]{\left|#1\right|}

\newcommand{\sub}{\subseteq}

\newcommand{\ot}{\otimes}
\newcommand{\otb}{\bar{\otimes}}

\newcommand{\rarrow}{\rightarrow}

\renewcommand{\iff}{\Leftrightarrow}

\DeclareSymbolFont{Symbols}{OMS}{cmsy}{m}{n}
\DeclareMathSymbol{\Emptyset}{\mathord}{Symbols}{"3B}

\newcommand{\act}{\curvearrowright}
\newcommand{\cross}{\rtimes}

\newcommand{\dint}{\int^\oplus}
\newcommand{\emb}{\prec}

\newcommand{\Aut}{\operatorname{Aut}}

\newcommand{\dom}{\operatorname{dom}}

\newcommand{\id}{\operatorname{id}}
\newcommand{\im}{\operatorname{im}}

\newcommand{\proj}{\operatorname{proj}}

\begin{document}

\title[Stable decompositions and rigidity for product equivalence relations]
{Stable decompositions and rigidity for products of countable equivalence relations}

\author[P. Spaas]{Pieter Spaas}
\address{Department of Mathematical Sciences, University of Copenhagen, Universitetsparken 5, DK-2100 Copenhagen \O, Denmark}
\email{pisp@math.ku.dk}

\begin{abstract} 
	We show that the ``stabilization'' of any countable ergodic p.m.p. equivalence relation which is not Schmidt, i.e. admits no central sequences in its full group, always gives rise to a stable equivalence relation with a unique stable decomposition, providing the first non-strongly ergodic such examples. In the proof, we moreover establish a new local characterization of the Schmidt property. We also prove some new structural results for product equivalence relations and orbit equivalence relations of diagonal product actions. 
\end{abstract}

\maketitle

\section{Introduction and statement of the main results}

Let $(X,\mu)$ be a standard probability space, and denote by $\Aut(X,\mu)$ its automorphism group, i.e. the group of Borel automorphisms of $X$ which preserve $\mu$ (and where we identify two such automorphisms if they agree $\mu$-almost everywhere). Given a countable discrete group $\Gamma$, a probability measure preserving (p.m.p.) action of $\Gamma$ on $(X,\mu)$ is a homomorphism from $\Gamma$ to $\Aut(X,\mu)$. Every such action generates an \textit{orbit equivalence relation} $\cR\coloneqq\cR(\Gamma\act X)\subset X\times X$, by letting $(x,y)\in\cR$ if and only if there exists $g\in\Gamma$ such that $y=gx$. In this case $\cR$ is a countable p.m.p. equivalence relation, i.e. every $\cR$-class is countable, and every Borel automorphism $\psi$ of $X$ for which $(\psi(x),x)\in\cR$ for almost every $x\in X$ is measure-preserving. In fact, every countable p.m.p. equivalence relation arises this way (see \cite{FM77}).

Let $\cR_{hyp}$ denote the (unique up to isomorphism) \textit{hyperfinite} ergodic p.m.p. equivalence relation. One of the main goals of this paper concerns the following general question: given countable ergodic p.m.p. equivalence relations $\cR$ and $\cS$, when are their ``stabilizations'' $\cR\times\cR_{hyp}$ and $\cS\times\cR_{hyp}$ isomorphic? To put this into context, recall from \cite{JS85} that a countable ergodic p.m.p. equivalence relation $\cR$ is called \textit{stable} if $\cR$ admits a decomposition $\cR\cong \cS\times\cR_{hyp}$ for some equivalence relation $\cS$. In recent years, there has been a lot of interest in the study of stable equivalence relations, see for instance \cite{Ki12,TD14,Ki16,Ma17,IS18}. In \cite{JS85}, Jones and Schmidt also establish a characterization of stability in terms of central sequences, providing a criterion to check when a decomposition as above exists. Below, we will be interested in the complementary question of when such a decomposition is unique. 

Following the terminology of \cite{IS18}, we say that a countable ergodic p.m.p. equivalence relation $\cR$ admits a \textit{stable decomposition} if $\cR$ can be written as $\cR = \cS\times\cR_{hyp}$ for some \textit{non-stable} equivalence relation $\cS$. We say $\cR$ admits a \textit{unique stable decomposition} if for any other stable decomposition $\cR = \cT\times\cR_{hyp}$, we necessarily have that $\cS$ and $\cT$ are stably isomorphic (see Section~\ref{ssec:eqrel} for the notion of stable isomorphism). Note that if $\cS$ and $\cT$ are stably isomorphic, we always have that $\cS\times\cR_{hyp}\cong \cT\times\cR_{hyp}$.

Motivated by the breakthrough result of Popa in \cite[Theorem~5.1]{Po06}, where he shows uniqueness of the McDuff decomposition of $N\otb R$ when $N$ is a non-Gamma II$_1$ factor and $R$ is the hyperfinite II$_1$ factor, the first unique stable decomposition result was proven in \cite[Theorem~G]{IS18}. There it is shown that $\cR\times\cR_{hyp}$ has a unique stable decomposition whenever $\cR$ is \textit{strongly ergodic}, i.e. has no asymptotically invariant sequences in the measure space (see Section~\ref{ssec:eqrel}). In this paper, we will instead study this problem in the absence of asymptotically central sequences in the full group $[\cR]$ of $\cR$, i.e. when $\cR$ is not Schmidt in the following sense.

\begin{definition}[cf. \cite{Sc87,KTD18}]
	Let $\cR$ be a countable ergodic p.m.p. equivalence relation on a standard probability space $(X,\mu)$. We say $\cR$ is \textit{Schmidt} if there exists a non-trivial central sequence in the full group $[\cR]$, i.e. a sequence $(T_n)_n\in [\cR]$ satisfying:
	\begin{enumerate}[nolistsep]
		\item (central) $\forall S\in [\cR]: \mu(\{x\in X\mid T_nS(x)\neq ST_n(x)\})\rarrow 0$, and
		\item (non-trivial) $\liminf_n \mu(\{x\in X\mid T_n(x)\neq x\})>0$.
	\end{enumerate}
\end{definition}

In Section~\ref{sec:Schmidt}, we will establish the following result, settling the unique stable decomposition problem for $\cR\times\cR_{hyp}$ when $\cR$ is not Schmidt. More precisely, we have:

\begin{mthm}\label{mthm:Schmidt}
	Suppose $\cR_1$ is a countable ergodic non-Schmidt p.m.p. equivalence relation on a standard probability space, and let $\cR_{hyp}$ be the hyperfinite ergodic p.m.p. equivalence relation. If $\cR_1\times\cR_{hyp}\cong \cR_2\times\cR_{hyp}$ for some countable ergodic p.m.p. equivalence relation $\cR_2$, then either 
	\begin{enumerate}[nolistsep]
		\item $\cR_2$ is also not Schmidt, and $\cR_1$ is stably isomorphic to $\cR_2$, or
		\item $\cR_2$ is stable, and thus $\cR_2\cong \cR_2\times\cR_{hyp}\cong \cR_1\times\cR_{hyp}$.
	\end{enumerate}
\end{mthm}

It is known that whenever a free ergodic p.m.p. action of a countable group $\Gamma$ is Schmidt in the above sense, then $\Gamma$ is necessarily inner amenable. Therefore, the above theorem applies for instance whenever $\cR_1$ is generated by a free ergodic p.m.p. action of any non-inner amenable group. Since every non-inner amenable group without property (T), for instance any nontrivial free product group, admits tons of free ergodic non-strongly ergodic p.m.p. actions, Theorem~\ref{mthm:Schmidt} provides a large class of new examples of equivalence relations with a unique stable decomposition. We also note here that the converse of the above observation is open, namely whether every inner amenable group admits a free ergodic p.m.p. action whose associated orbit equivalence relation is Schmidt. 

\begin{remark}
	Together with \cite[Theorem~G]{IS18}, Theorem~\ref{mthm:Schmidt} thus shows that equivalence relations which either lack central sequences in the measure space (i.e. strongly ergodic ones) or lack central sequences in the full group (i.e. non-Schmidt ones), admit unique stable decompositions. In other words, some ``control'' over the central sequences allows for unique decomposition results. Even though they are technically much more involved, both these results can therefore be seen as analogues for equivalence relations of the aforementioned unique McDuff decomposition result \cite[Theorem~5.1]{Po06}. Also, they can be compared to further  results about uniqueness of McDuff decompositions of von Neumann algebras as in \cite[Theorems~B, D]{IS18}.  
\end{remark}

There are a few ingredients that go into the proof of Theorem~\ref{mthm:Schmidt}. Firstly, we prove a new local characterization of the Schmidt property in Lemma~\ref{lem:local}, which leads to a spectral gap criterion for non-Schmidt equivalence relations. We will use this to characterize in Lemma~\ref{lem:cseqSchmidt} the central sequences in the full group of $\cR\times\cR_{hyp}$ when $\cR$ is not Schmidt. Secondly, we will use this characterization to ``detect'' from $\cR\times\cR_{hyp}$ the absence of the Schmidt property for $\cR$ (Lemma~\ref{lem:nonSchmidt}), and to establish certain intertwining results for equivalence relations of this type (Lemma~\ref{lem:Schmidt}). A combination of these ingredients and Theorem~\ref{mthm:int}(2) below, will then lead to the desired conclusion.

Theorem~\ref{mthm:int} will arise in the general framework of \textit{intertwining techniques} for two subequivalence relations $\cS,\cT$ of a countable p.m.p. equivalence relation $\cR$, introduced in \cite{Io11} as an analogue of Popa's intertwining-by-bimodules (see Section~\ref{ssec:Popa}), and further investigated and used in \cite{DHI16}. We will further study and develop these techniques in Section~\ref{sec:int}. Intuitively, $\cT$ intertwines into $\cS$ if---upon restricting to appropriate subsets of $X$---there exist a subequivalence relation of bounded index (see Section~\ref{ssec:finiteindex}) $\cT_0\leq \cT$ and an element $\theta\in [[\cR]]$ such that $(\theta\times\theta)(\cT_0)$ is a subequivalence relation of $\cS$. By analogy with Popa's intertwining-by-bimodules in the von Neumann algebra setting, Ioana characterized this intertwining in \cite[Lemma~1.7]{Io11} using the function $\varphi_{\cS}:[[\cR]]\rarrow [0,1]$ defined by
\[
\varphi_{\cS}(\theta) = \mu(\{x\in \dom(\theta)\mid (\theta(x),x)\in \cS\}).
\]
Our next main Theorem shows that for product equivalence relations, the presence of intertwining leads to certain rigidity phenomena. We refer to Propositions~\ref{prop:intertwineprod} and \ref{prop:commutantish} for the proofs and slightly more general statements, and to Definition~\ref{def:inteqrel} for the notation.

\begin{mthm}\label{mthm:int}
	Let $\cR_1$, $\cS_1$, $\cR_2$, $\cS_2$ be countable ergodic p.m.p. equivalence relations on standard probability spaces 
	such that $\cR\coloneqq\cR_1\times\cS_1\cong\cR_2\times\cS_2$.
	\begin{enumerate}[nolistsep]
		\item If, as subequivalence relations of $\cR$, we have $\cR_1\emb_\cR\cR_2$, then there exists an equivalence relation $\cT$ such that $\cR_2$ is stably isomorphic to $\cR_1\times \cT$.
		\item If also $\cR_2\emb_\cR\cR_1$, then $\cR_1$ is stably isomorphic to $\cR_2$, and $\cS_1$ is stably isomorphic to $\cS_2$.
	\end{enumerate}
\end{mthm}

The proofs of these results are inspired by the proof of the aforementioned \cite[Theorem~G]{IS18}, where it is shown that $\cR\times\cR_{hyp}$ admits a unique stable decomposition when $\cR$ is strongly ergodic. In the language of this paper, the proof essentially shows that the strong ergodicity of $\cR$ implies the intertwining condition in Theorem~\ref{mthm:int}(1).

Next, we turn our attention to equivalence relations arising as orbit equivalence relations of diagonal product actions of countable discrete groups. On top of the product equivalence relations from Theorem~\ref{mthm:int}, we will be able to apply the intertwining techniques in this setting as well. More specifically, we will study actions of the following type.

\begin{restatable}{notation}{restatenotation}\label{not:notation}
	Let $\Gamma_1$, $\Gamma_2$, $\Sigma_1$, $\Sigma_2$ be countable discrete groups, and $\rho_1:\Gamma_1\rarrow\Sigma_1$, $\rho_2:\Gamma_2\rarrow\Sigma_2$ be surjective group homomorphisms. Denote $\Gamma_i^0\coloneqq \ker(\rho_i)$. Suppose $\Gamma_i\act (X_i,\mu_i)$ and $\Sigma_i\act (Y_i,\nu_i)$ are free ergodic p.m.p. actions on standard probability spaces, and assume that also $\Gamma_i^0\act (X_i,\mu_i)$ is ergodic for $i=1,2$.
	We then have the diagonal product actions $\Gamma_i\act (X_i,\mu_i)\times (Y_i,\nu_i)$ given by $g\cdot(x,y) = (gx,\rho(g)y)$, for all $g\in\Gamma_i$, $x\in X_i$, $y\in Y_i$. For $i=1,2$, we will write $\cR_i\coloneqq\cR(\Gamma_i^0\act X_i)$.
\end{restatable}

Our last main result shows that, when taking such a product of a strongly ergodic action with an amenable action, we can ``recover'' parts of the strongly ergodic piece.

\begin{mthm}\label{mthm:SOE}
	In the setting of Notation~\ref{not:notation}, assume the following extra conditions for $i=1,2:$
	\begin{enumerate}[label=(\roman*),nolistsep]
		\item $\Gamma_i\act (X_i,\mu_i)$ is strongly ergodic,
		\item $\Sigma_i$ is amenable.
	\end{enumerate}
	If the actions $\Gamma_1\act (X_1\times Y_1,\mu_1\times\nu_1)$ and $\Gamma_2\act (X_2\times Y_2,\mu_2\times\nu_2)$ are orbit equivalent, then $\Gamma_1^0\act X_1$ and $\Gamma_2^0\act X_2$ are necessarily stably orbit equivalent.
\end{mthm}

\begin{remark}
	This class of actions is also studied in \cite{Io06a}, where Ioana establishes several results that are similar in spirit to Theorem~\ref{mthm:SOE}, see for instance \cite[Proposition~2.4, Proposition~3.3, and Theorem~4.1]{Io06a}.
\end{remark}

We will end this introduction with several observations regarding the relation of orbit equivalence on the space $A(\Gamma,X,\mu)$ of p.m.p. actions of a fixed countable discrete group $\Gamma$ on a standard probability space $(X,\mu)$. We refer for instance to \cite{Ke10} for a detailed study of this space. Building on earlier work by Ioana and Epstein (\!\!\cite{Io06b,Ep07}), Epstein, Ioana, Kechris, and Tsankov showed in \cite{IKT08} that for any non-amenable group $\Gamma$, free ergodic (in fact mixing) p.m.p. actions of $\Gamma$ on $(X,\mu)$ are not classifiable by countable structures. Even though this settles the non-classifiability of actions of any non-amenable group up to orbit equivalence, it might be that certain natural subclasses of actions admit a classification. In particular, it is not clear from the co-induction construction used to establish this result, whether the actions constructed in \cite{IKT08} are in general strongly ergodic or not. 

Theorem~\ref{mthm:SOE} above can be used to show Theorem~\ref{mthm:actions} below for a large class of groups, by associating to any strongly ergodic action in a suitable ``unclassifiable'' family (see also Theorem~\ref{mthm:strong}), a non-strongly ergodic action via a diagonal product construction as in Notation~\ref{not:notation}. However, we observed that the more general statement of Theorem~\ref{mthm:actions} below follows from a similar, albeit easier, adaptation of the construction in \cite{IKT08}, without having to appeal to Theorem~\ref{mthm:SOE} above. In particular, the proofs of the remaining results in this introduction will largely be deduced from the existing literature. Nevertheless, even though some of them may have been observed by experts, they seem not to have appeared in writing as of now, and so we collect them here.

Firstly, let $\Gamma$ be any non-amenable group without property (T). Using the same proof as \cite[Theorem~3.12]{IKT08}, one can establish the following result about the classification of \textit{non-strongly ergodic} actions of $\Gamma$. Note that if $\Gamma$ has property (T), this question would be void, since it is well-known that in this case all ergodic actions of $\Gamma$ are strongly ergodic (see \cite{Sc80}).

\begin{theorem}[\!\!\cite{IKT08}]\label{mthm:actions}
	Let $\Gamma$ be any non-amenable countable discrete group without property (T). Then free ergodic non-strongly ergodic p.m.p. actions of $\Gamma$ up to orbit equivalence are not classifiable by countable structures.
\end{theorem}
\begin{proof}[Sketch of the proof]
Fix any non-strongly ergodic action $a$ of $\Gamma$, and consider the family of actions $b(\pi_i)$ constructed in \cite[Section~3]{IKT08}. It is then easy to check that the diagonal product actions $b(\pi_i)\times a$, which are obviously non-strongly ergodic, still satisfy the necessary rigidity properties in order to apply \cite[Lemma~3.13]{IKT08}, and then the proof can be finished in exactly the same way.
\end{proof}

\begin{remark}
	Using a co-induction construction similar to \cite{IKT08}, it was shown in \cite{GL17} that the equivalence relations of (stable) orbit equivalence and (stable) von Neumann equivalence of free ergodic p.m.p. actions of a non-amenable group are not Borel. Using a similar diagonal product construction as above, one can show that this result also holds for free ergodic non-strongly ergodic actions of a non-property (T) group which contains a free group.
\end{remark}

In \cite{ST09}, it is shown that the isomorphism relation of II$_1$ factors is not classifiable by countable structures. The II$_1$ factors constructed to establish this result arise from actions of the free group $\bF_3$. As observed in \cite{Sa16}, these actions are strongly ergodic, and thus, since $\bF_3$ is not inner amenable, these II$_1$ factors do not have property Gamma (cf. \cite{Ch82}). Furthermore, using Popa's unique McDuff decomposition result in the absence of property Gamma (\!\!\cite[Theorem~5.1]{Po06}), it is  observed in \cite{Sa16} that therefore McDuff II$_1$ factors up to isomorphism are not classifiable by countable structures either. Using Theorem~\ref{mthm:actions}, we can deduce that the same holds for II$_1$ factors with property Gamma that are not McDuff:

\begin{corollary}\label{mcor:Gamma}
	The equivalence relation of isomorphism on the space of non-McDuff II$_1$ factors with property Gamma is not classifiable by countable structures.
\end{corollary}
\begin{proof}
Consider for instance the free group $\bF_2$. By Theorem~\ref{mthm:actions}, non-strongly ergodic actions of $\bF_2$ are not classifiable by countable structures. Now from \cite{Si55} and the uniqueness of Cartan subalgebras \cite[Theorem~1.2]{PV11}, it follows that the associated group measure space II$_1$ factors are not classifiable by countable structures either. Moreover, since the actions are non-strongly ergodic, these II$_1$ factors have property Gamma. Finally, since $\bF_2$ is not inner amenable, it follows from \cite{Ch82} that they are not McDuff.
\end{proof}
 
Next, together with a recent result from Drimbe \cite[Proposition~B]{Dr20}, one can show that for many groups, the actions constructed in \cite{IKT08} can be made strongly ergodic.

\begin{theorem}[{\!\!\cite{IKT08,Dr20}}]\label{mthm:strong}
	Let $\Gamma$ be one of the following.
	\begin{enumerate}[nolistsep]
		\item A group with property (T). 
		\item More generally, a group without the Haagerup property.
		\item A group which contains a free subgroup which is not co-amenable in $\Gamma$ (see, e.g., \cite{Ey72}).
	\end{enumerate}
	Then strongly ergodic actions of $\Gamma$ up to orbit equivalence are not classifiable by countable structures.
\end{theorem}
\begin{proof}
If $\Gamma$ has property (T), then every ergodic action is strongly ergodic, and more generally, if $\Gamma$ does not have the Haagerup property, then  every mixing action of $\Gamma$ is strongly ergodic by a result of Jolissaint (\!\!\!\cite[Chapter~2]{CCJJV}, see also \cite[Theorem~11.1]{Ke10}). In both cases, the result thus follows from \cite[Theorem~5]{IKT08}. If $\Gamma$ contains a non-co-amenable free group $\bF_n$, then from \cite{Io06b,IKT08}, we first of all get a family of actions of $\bF_n$ such that the family of associated co-induced actions of $\Gamma$ are not classifiable by countable structures. Using \cite[Proposition~B]{Dr20} and the fact that $\bF_n$ is not co-amenable in $\Gamma$, it moreover follows that all these actions are strongly ergodic.
\end{proof}

Finally, together with \cite[Theorem~G]{IS18} and Theorem~\ref{mthm:Schmidt}, we can deduce the following classification result for stable actions, where we say that an action is \textit{stable} if its associated orbit equivalence relation is stable.

\begin{corollary}\label{mcor:StableActions}
	Let $\Gamma$ either be a group as in Theorem~\ref{mthm:strong}, or any non-inner amenable group, and let $\Sigma$ be an infinite amenable group. Then stable actions of $\Gamma\times\Sigma$ up to orbit equivalence are not classifiable by countable structures.
\end{corollary}
\begin{proof}
First, assume $\Gamma$ is as in Theorem~\ref{mthm:strong}. Fix a free ergodic p.m.p. action $a$ of $\Sigma$ on some standard probability space $(Y,\nu)$, and let $(b_i)_{i\in I}$ be the family of pairwise non-orbit equivalent strongly ergodic actions of $\Gamma$ from Theorem~\ref{mthm:strong}. Using the rigidity result \cite[Lemma~7.4]{BHI15} instead of \cite[Lemma~3.13]{IKT08} in the proof, one can show that the actions $b_i$ are in fact pairwise not stably orbit equivalent. We claim that then the $b_i\times a$ are pairwise not orbit equivalent. Indeed, if $b_i\times a$ would be orbit equivalent to $b_j\times a$ for some $i\neq j$, it would follow from \cite[Theorem~G]{IS18} that $b_i$ is stably orbit equivalent to $b_j$, contradiction.

Second, if $\Gamma$ is not inner amenable, we can take any family of pairwise non-orbit equivalent free ergodic p.m.p. actions of $\Gamma$ (for instance from \cite{Ep07}) witnessing non-classifiability by countable structures. Observing that these actions are necessarily not Schmidt, we can thus copy the first part of the proof using Theorem~\ref{mthm:Schmidt} instead of \cite[Theorem~G]{IS18}.
\end{proof}

\subsection*{Organization of the paper} Besides the Introduction, there are three other sections in this paper. In Section~\ref{sec:prelim} we collect some necessary preliminaries. In Section~\ref{sec:int} we discuss intertwining techniques for equivalence relations, and we prove Theorems~\ref{mthm:int} and \ref{mthm:SOE}. Finally in Section~\ref{sec:Schmidt} we discuss non-Schmidt equivalence relations and prove Theorem~\ref{mthm:Schmidt}.

\subsection*{Acknowledgments} I am very grateful to both Andrew Marks and Adrian Ioana for several stimulating discussions about results and topics in or related to this paper. I would also like to thank Adrian Ioana for several useful comments on an earlier draft of this paper, and the anonymous referee for pointing out a few gaps in an earlier version, and suggesting several improvements.

\section{Preliminaries}\label{sec:prelim}

\subsection{Equivalence relations}\label{ssec:eqrel}
Suppose $(X,\mu)$ is a standard probability space, and $\Gamma\act X$ is a p.m.p. action of a countable group $\Gamma$. Denote by $\cR(\Gamma\act X)\coloneqq \{(x,y)\in X\times X\mid \exists g\in\Gamma: g\cdot x = y\}$ its \textit{orbit equivalence relation}. By construction, this is a countable p.m.p. equivalence relation. Conversely, every countable p.m.p. equivalence relation is of this form, see \cite{FM77}.

Given any countable p.m.p. equivalence relation $\cR$ on $(X,\mu)$, we denote by $[x]_\cR$ the equivalence class of a point $x\in X$. We endow $\cR\subset X\times X$ with the (infinite) measure $\bar{\mu}$ given by
\[
\bar{\mu}(A) = \int_X \#\{y\in X\mid (x,y)\in A\}\,d\mu(x),
\]
for every Borel set $A\subset X\times X$, and where $\#$ denotes the counting measure. The automorphism group $\Aut(\cR)$ of $\cR$ is the set of all measure space automorphisms $\al$ of $(X,\mu)$ such that $(\al(x),\al(y))\in\cR$ for $\bar{\mu}$-almost every $(x,y)\in\cR$. The \textit{full group} $[\cR]$ of $\cR$ is the subgroup of all $\al\in \Aut(\cR)$ such that $(x,\al(x))\in\cR$ for $\mu$-almost every $x\in X$. The \textit{full pseudogroup} $[[\cR]]$ of $\cR$ consists of all isomorphisms $\al:(A,\mu_A)\rarrow (B,\mu_B)$ satisfying $(x,\al(x))\in\cR$ for almost every $x\in A$, where $\mu_A$ and $\mu_B$ denote the restrictions of $\mu$ to the measurable subsets $A,B\subset X$.

We call $\cR$ \textit{ergodic} if every measurable set $A\subset X$ such that $\mu(\alpha(A)\Delta A)=0$ for all $\al\in [\cR]$, satisfies $\mu(A)\in\{0,1\}$. We say $\cR$ is \textit{strongly ergodic} if for every asymptotically invariant sequence of measurable sets $A_n\sub X$, i.e. satisfying $\lim\limits_{n\rarrow\infty} \mu(\alpha(A_n)\Delta A_n)=0$ for all $\alpha\in [\cR]$, we have $\lim\limits_{n\rarrow\infty} \mu(A_n)(1-\mu(A_n))=0$. Similarly, the action $\Gamma\act X$ is called ergodic, respectively strongly ergodic, if these statements hold with $\alpha\in [\cR]$ replaced by $g\in\Gamma$. We note that a p.m.p. action $\Gamma\act (X,\mu)$ is (strongly) ergodic if and only if its orbit equivalence relation $\cR(\Gamma\act X)$ is (strongly) ergodic.

Assuming $\cR$ is ergodic, we can define the $t$\textit{-amplification} $\cR^t$ of $\cR$ where $t>0$ is any real number. Let $\#$ denote the counting measure on $\N$ and consider a measurable set $X^t\subset X\times \N$ with $(\mu\times\#)(X^t)=t$. Then $\cR^t$ is defined as the equivalence relation on $X^t$ given by $((x,m),(y,n))\in\cR^t$ if and only if $(x,y)\in \cR$. Since $\cR$ is ergodic, this defines $\cR^t$ uniquely up to isomorphism.

We call two countable ergodic p.m.p. equivalence relations $\cR$ and $\cS$ \textit{stably isomorphic} if there exists $t>0$ such that $\cR\cong \cS^t$. Similarly, we say two ergodic p.m.p. actions $\Gamma\act (X,\mu)$ and $\Lambda\act (Y,\nu)$ are \textit{(stably) orbit equivalent}, if their associated orbit equivalence relations are (stably) isomorphic. Finally, we note that for any $t>0$ and $\cR$, $\cS$ as above, we have isomorphisms $\cR\times \cS\cong \cR^t\times\cS^{1/t}$.

\subsection{Tracial von Neumann algebras}

In this paper we will also work with \textit{tracial von Neumann algebras} $(M,\tau)$, i.e. $M$ is a von Neumann algebra equipped with a faithful normal tracial state $\tau:M\rarrow\C$. We denote by $\norm{x}_2 \coloneqq \sqrt{\tau(x^*x)}$ the \textit{2-norm} of $x$ and denote by $L^2(M)$ the completion of $M$ with respect to this norm. Unless stated otherwise, we will always assume $L^2(M)$ to be a separable Hilbert space, in which case we also call $M$ a separable von Neumann algebra. We will further denote by $\cU(M)$ the unitary group of $M$, and by $\Aut(M)$ the group of $\tau$-preserving automorphisms of $M$ equipped with the Polish topology of pointwise $\norm{.}_2$-convergence.

Let $P\subset M$ be a von Neumann subalgebra, which unless stated otherwise is assumed to be unital. We denote by $E_P:M\rarrow P$ the unique $\tau$-preserving conditional expectation from $M$ onto $P$. Given another subalgebra $Q\subset M$ and $\eps>0$, we write $Q\subset_{\eps}P$ if we have  $\norm{x-E_P(x)}_2\leq\eps$, for every $x\in (Q)_1$. We further denote by $P'\cap M \coloneqq \{x\in M\mid \forall y\in P: xy=yx\}$ the relative commutant of $P$ in $M$, and by $\cN_M(P)\coloneqq \{u\in\cU(M)\mid uPu^*=P\}$ the normalizer of $P$ in $M$. We say that $P$ is regular in $M$ if $\cN_M(P)''=M$, i.e. $\cN_M(P)$ generates $M$ as a von Neumann algebra. Most von Neumann algebras $M$ we encounter will be factors, i.e. the center $\cZ(M)\coloneqq M'\cap M$ satisfies $\cZ(M)=\C 1$.

\subsection{Cartan subalgebras and the von Neumann algebra of an equivalence relation}

Given a countable p.m.p. equivalence relation $\cR$, we can associate a canonical von Neumann algebra $L(\cR)\sub B(L^2(\cR,\bar{\mu}))$ to it (see \cite{FM77}). This von Neumann algebra is generated by partial isometries $u_\varphi$ for $\varphi\in [[\cR]]$, and contains a canonical copy of $L^\infty(X)$ as a Cartan subalgebra, i.e. a maximal abelian regular von Neumann subalgebra. This von Neumann algebra $L(\cR)$ is a II$_1$ factor if and only if $\cR$ is ergodic. 

Conversely, if $(M,\tau)$ is a separable tracial von Neumann algebra and $A\subset M$ is an abelian von Neumann subalgebra, we get a countable p.m.p. equivalence relation in the following way. Identify $A=L^\infty(X)$, for some standard probability space $(X,\mu)$. Then for every $u\in\cN_M(A)$, we can find an automorphism $\alpha_u$ of $(X,\mu)$ such that $a\circ\alpha_u=uau^*$, for every $a\in A$. The equivalence relation $\cR(A\subset M)$ of the inclusion $A\subset M$ is then defined to be the smallest countable p.m.p. equivalence relation on $(X,\mu)$ whose full group contains $\alpha_u$, for every $u\in\cN_M(A)$. 

Now, assume that $M$ is a II$_1$ factor and $A\subset M$ is a Cartan subalgebra. Then $\cR(A\subset M)$ is ergodic. Moreover, if $(A\subset M)\cong (L^\infty(X)\subset L(\cR))$ for some countable ergodic p.m.p. equivalence relation $\cR$, then $\cR(A\subset M)\cong \cR$, i.e. we can recover the equivalence relation $\cR$ from the von Neumann algebra inclusion $(L^\infty(X)\subset L(\cR))$. Furthermore, one can show that any Cartan inclusion $(A\subset M)$ arises as an inclusion $(L^\infty(X)\subset L_w(\cR))$ of $L^\infty(X)$ inside the ``twisted'' von Neumann algebra $L_w(\cR)$ for some 2-cocycle $w\in$ H$^2(\cR,\bT)$, see \cite{FM77}.  

For $t>0$, we define the $t$-amplification of a II$_1$ factor $M$, denoted by $M^t$, as the isomorphism class of $p(B(\ell^2)\otb M)p$, where $p\in B(\ell^2)\otb M$ is a projection with $(\text{Tr}\ot\tau)(p)=t$, and $\text{Tr}$ is the usual trace on $B(\ell^2)$. Similarly, the inclusion $(A^t\subset M^t)$ is defined as the isomorphism class of the inclusion $(p(\ell^\infty\ot A)p\subset p(B(\ell^2)\otb M)p)$, where  $p\in B(\ell^2)\otb A$ is a projection with $(\text{Tr}\ot\tau)(p)=t$, and $\ell^\infty\subset B(\ell^2)$ is the subalgebra of diagonal operators. With this notation, we then have that $A^t\subset M^t$ is a Cartan subalgebra, and $\cR(A^t\subset M^t)\cong\cR(A\subset M)^t$.

\subsection{Popa's intertwining-by-bimodules}\label{ssec:Popa}
In \cite{Po03}, Popa introduced a powerful theory for deducing unitary conjugacy of subalgebras of tracial von Neumann algebras, which we briefly recall here.

Let $P,Q$ be von Neumann subalgebras of a tracial von Neumann algebra $(M,\tau)$. We say that \textit{a corner of $P$ embeds into $Q$ inside $M$} (or, \textit{$P$ intertwines into $Q$ inside $M$}), and write $P\emb_M Q$ if we can find non-zero projections $p\in P$, $q\in Q$, a $*$-homomorphism $\theta:pPp\rarrow qQq$, and a non-zero partial isometry $v\in qMp$ satisfying $\theta(x)v=vx$, for all $x\in pPp$. Moreover, if $Pp'\emb_M Q$, for every non-zero projection $p'\in P'\cap M$, we write $P\emb_M^s Q$, and say that \textit{$P$ strongly intertwines into $Q$ inside $M$}. The main technical tool from Popa's deformation/rigidity theory is the following characterization of intertwining.

\begin{theorem}[\!\!\cite{Po03}]\label{thm:Popaint}
	Let $P,Q$ be von Neumann subalgebras of a tracial von Neumann algebra $(M,\tau)$, and let $\cU\subset\cU(P)$ be a subgroup which generates $P$. Then the following conditions are equivalent:
	\begin{enumerate}
		\item $P\emb_M Q$.
		\item There is no sequence $u_n\in \cU$ satisfying $\norm{E_Q(au_nb)}_2\rarrow 0$, for all $a,b\in M$.	
	\end{enumerate} 
\end{theorem}

\subsection{Finite index subequivalence relations}\label{ssec:finiteindex}
In Section~\ref{sec:int}, we will discuss Ioana's intertwining techniques for equivalence relations, which were inspired by Popa's aforementioned techniques for von Neumann algebras. For this, we will have to deal with subequivalence relations of bounded index, which we briefly recall here.

Let $\cR$ be a countable p.m.p. equivalence relation on a standard probability space $(X,\mu)$, and suppose $\cT\leq \cR$ is a subequivalence relation. Then we can decompose $X=\sqcup_{N\in\N\cup\{\infty\}} X_N$, where for every $N\in\N\cup\{\infty\}$,
\[
X_N\coloneqq \{x\in X\mid [x]_\cR \text{ is the union of } N\; \cT\text{-classes}\}
\]
is the $\cR$-invariant set consisting of all points in $X$ whose $\cR$-class contains exactly $N$ $\cT$-classes.

If $\mu(X_\infty)=0$, we say that the inclusion $\cT\leq\cR$ has \textit{(essentially) finite index}. If there exists $k\geq 1$ such that $\mu(X_N)=0$ for all $N>k$, we say that the inclusion $\cT\leq\cR$ has \textit{bounded index}.

For future reference, we record the following Lemma, which follows almost immediately from \cite[Lemma~3.3]{DHI16}.

\begin{lemma}\label{lem:DHIprodbddindex}
	Let $\cR$ on $(X,\mu)=(X_1\times X_2,\mu_1\times\mu_2)$ be a product of two countable p.m.p. equivalence relations, $\cR=\cR_1\times\cR_2$. Suppose $Y\subset X$ is a subset of positive measure, $\cT\leq (\cR_1\times\id)\rvert_Y$ is a subequivalence relation, and $\theta\in [[\cR]]$ with $\dom(\theta)=Y$ satisfies $(\theta\times\theta)(\cT)\leq (\cR_1\times\id)\rvert_{\theta(Y)}$.
	
	If $\cT\leq (\cR_1\times\id)\rvert_Y$ has bounded (respectively, essentially finite) index, then there is a sequence of $\cT$-invariant positive measure subsets $Y_n\subset Y$ with $Y=\cup_{n=1}^\infty Y_n$ such that $(\theta\times\theta)(\cT\rvert_{Y_n})\leq (\cR_1\times\id)\rvert_{\theta(Y_n)}$ has bounded (respectively, essentially finite) index for each $n\geq 1$.
\end{lemma}
\begin{proof}
	Since $\cR_1$ on $X_1$ and $\cR_2$ on $X_2$ are countable p.m.p. equivalence relations, by \cite{FM77} we can write $\cR_i=\cR(\Gamma_i\act X_i)$ for $i=1,2$. Hence $\cR=\cR(\Gamma_1\times\Gamma_2\act X)$ and $\cR_1\times\id = \cR(\Gamma_1\act X)$. We can thus apply \cite[Lemma~3.3]{DHI16}, and the result follows.
\end{proof}

\section{Intertwining equivalence relations}\label{sec:int}

Let $\cR$ be a countable p.m.p. equivalence relation on the standard probability space $(X,\mu)$. Consider a subequivalence relation $\cS\leq\cR$ on $(X,\mu)$ such that every $\cR$-class contains infinitely many $\cS$-classes. Following \cite{IKT08} we define a function $\varphi_\cS:[[\cR]]\rarrow [0,1]$ by
\[
\varphi_\cS(\theta)=\mu(\{x\in\dom(\theta)\mid (\theta(x),x)\in\cS\}).
\]

\begin{lemma}[\!\!\!{\cite[Lemma~1.7]{Io11}}]\label{lem:inteqrel}
	Let $E\sub X$ be a positive measure subset and $\cT\leq\cR\rvert_E$ a subequivalence relation. Then the following are equivalent.
	\begin{enumerate}
		\item There is no sequence $\{\theta_n\}_{n=1}^\infty\sub [\cT]$ such that $\varphi_\cS(\psi\theta_n\psi')\rarrow 0$ for all $\psi,\psi'\in [\cR]$.
		\item There exist a $\cT$-invariant subset $E'\sub E$ of positive measure and a subequivalence relation $\cT_0\leq \cT$ such that for any positive measure subset $E_0\sub E'$, there is a positive measure subset $Y\sub E_0$ and $\theta\in [[\cR]], \theta:Y\rarrow Z$, such that
		\begin{enumerate}[label=(\alph*)]
			\item $\cT_0\rvert_Y\leq \cT\rvert_Y$ has bounded index, and
			\item $(\theta\times\theta)(\cT_0\rvert_Y)\leq \cS\rvert_Z$.
		\end{enumerate}
	\end{enumerate}
\end{lemma}

\begin{definition}\label{def:inteqrel}
	Whenever the equivalent conditions from Lemma~\ref{lem:inteqrel} hold, we will write $\cT\emb_\cR\cS$, and say that \textit{$\cT$ intertwines into $\cS$ inside $\cR$}.
\end{definition}

The above Lemma is phrased entirely in the context of equivalence relations. However, as indicated before, it was discovered in the context of Popa's intertwining-by-bimodules techniques, from which we also borrowed the above notation and terminology. Moreover, there is the following direct correspondence between the two frameworks.

\begin{lemma}[\!\!\!{\cite[Lemma~1.8]{Io11}}]\label{lem:RvLR}
	Let $E\sub X$ be a positive measure subset and $\cT\leq\cR\rvert_E$ a subequivalence relation. Then the following are equivalent.
	\begin{enumerate}
		\item $\cT\emb_\cR \cS$, and
		\item $L(\cT)\emb_{L(\cR)} L(\cS)$.
	\end{enumerate}
\end{lemma}

Note that each of the von Neumann algebras $L(\cT)$, $L(\cS)$, and $L(\cR)$ contains the canonical Cartan subalgebra $L^\infty(X)$ (cut down by the projection $1_E$ where necessary). In the product setting, this gives the following.

\begin{corollary}\label{cor:RvLRprod}
	Let $\cR_1, \cS_1, \cR_2, \cS_2$ be countable ergodic p.m.p. equivalence relations, on standard probability spaces $X_1$, $Y_1$, $X_2$, $Y_2$ respectively, such that $\cR\coloneqq\cR_1\times\cS_1\cong\cR_2\times\cS_2$. Then the following are equivalent, where we identify $\cR_i \coloneqq \cR_i\times\id_{Y_i}\leq \cR$
	\begin{enumerate}
		\item $\cR_1\emb_\cR\cR_2$,
		\item $L(\cR_1)\otb L^\infty(Y_1)\emb_{L(\cR)} L(\cR_2)\otb L^\infty(Y_2)$,
		\item $L^\infty(Y_2)\emb_{L(\cR)} L^\infty(Y_1)$.
	\end{enumerate}
\end{corollary}
\begin{proof}
	The equivalence between (1) and (2) is immediate from Lemma~\ref{lem:RvLR}. The equivalence between (2) and (3) follows from \cite[Lemma~3.5]{Va08}.
\end{proof}

In the von Neumann algebra setting, one can generally deduce unitary conjugacy from intertwining between different tensor factors of a given II$_1$ factor, see for instance \cite[Proposition~12]{OP03}. However, for equivalence relations, more work needs to be done. Indeed, from the above Corollary we see that the intertwining $\cR_1\emb_\cR \cR_2$ is not equivalent to $L(\cR_1)\emb_{L(\cR)} L(\cR_2)$. Rather, one also has to take the Cartan subalgebras $L^\infty(Y_i)$ into account, which leads to a much different analysis. 

Note for instance that for the equivalence between (2) and (3) in the above Corollary, we could apply the very useful result \cite[Lemma~3.5]{Va08} about the intertwining of relative commutants in the von Neumann algebra setting. Unfortunately, there is no direct analogue of this for the intertwining of equivalence relations. Nevertheless, we will be able to deduce some results of a similar flavor in the product setting (see for instance Proposition~\ref{prop:commutantish} below). 

For the remainder of this Section, we will turn to the study of intertwining results for product equivalence relations, and orbit equivalence relations arising from diagonal product actions. Firstly, we record the following easy Lemma, see for instance \cite[Lemma~7.1]{IS18}.

\begin{lemma}\label{lem:split}
	Suppose $M$, $N$ are II$_1$ factors, and $A\subset M$, $B\subset N$ are Cartan subalgebras. Then
	\[
	\cR(A\subset M)\times \cR(B\subset N) \cong \cR(A\otb B\subset M\otb N).
	\]
\end{lemma}

Given an intertwining between factors of product equivalence relations, we can now show the following.

\begin{proposition}\label{prop:intertwineprod}
	Let $\cR_1$, $\cS_1$, $\cR_2$, $\cS_2$ be countable ergodic p.m.p. equivalence relations on standard probability spaces $(X_1,\mu_1)$, $(Y_1,\nu_1)$, $(X_2,\mu_2)$, $(Y_2,\nu_2)$ respectively such that $\cR\coloneqq\cR_1\times\cS_1\cong\cR_2\times\cS_2$. Assume that, as subequivalence relations of $\cR$, we have $\cR_1\emb_\cR \cR_2$.
	
	Then there exists an equivalence relation $\cT$ such that $\cR_2$ is stably isomorphic to $\cR_1\times \cT$. Moreover, if $\cS_1$ is hyperfinite, then either $\cR_2$ is stably isomorphic to $\cR_1$, or $\cR_2\cong\cR_1\times\cS_1\cong\cR_2\times\cS_2$.
\end{proposition}
\begin{proof}
By Lemma~\ref{lem:inteqrel} we can find a subequivalence relation $\cT_0\leq \cR_1 = \cR_1\times \id_{Y_1}$, positive measure subsets $W,Z\sub X\coloneqq X_1\times Y_1$, and $\theta\in [[\cR]], \theta:W\rarrow Z$, such that
\begin{enumerate}[label=(\alph*)]
	\item $\cT_0\rvert_W\leq \cR_1\rvert_W$ has bounded index, say bounded by $k$, and
	\item $(\theta\times\theta)(\cT_0\rvert_W)\leq \cR_2\rvert_Z$.
\end{enumerate}
Note that, as a subequivalence relation of $\cR$, $\cR_1$ is a fibered equivalence relation over $Y_1$:
\[
\cR_1\times \id_{Y_1} = \dint_{Y_1} \cR_1 \,d\nu_1(y),
\]
where inside the integral we look at $\cR_1$ as an equivalence relation on $X_1$. We note that this is in fact the ergodic decomposition of $\cR_1\times \id_{Y_1}$. Given the subequivalence relation $\cT_0\leq \cR_1\times \id_{Y_1}$, we can thus write
\[
\cT_0 = \dint_{Y_1} \cT_{0,y} \,d\nu_1(y).
\]
Here for every $y\in Y_1$, $\cT_{0,y}$ is an equivalence relation on $X_1$. By (a) above, we note that moreover $\cT_{0,y}\rvert_{W_y}$ has bounded index at most $k$ inside $\cR_1\rvert_{W_y}$ for almost every $y\in Y_1$, where $W_y=\{x\in X_1\mid (x,y)\in W\}$. Since $\cR_1$ is ergodic, this implies in particular that for almost every $y\in Y_1$, the ergodic decomposition of $\cT_{0,y}\rvert_{W_y}$ is atomic with at most $k$ atoms. By restricting to further subsets if necessary, we can thus assume that for almost every $y$, $\cT_{0,y}$ is ergodic on $W_y$. Taking yet further subsets if necessary, we can also assume that there is some $c>0$ such that for every $y\in Y_1$, we have either $\mu_1(W_y)=c$ or $\mu_1(W_y)=0$. 

Fix any measurable set $C\sub X_1$ with $\mu_1(C)=c$. By ergodicity of $\cR_1$ and a standard uniformization argument, it follows that there is a measurable map $\psi:Y_1\rarrow [\cR_1]$ such that $\psi_y(W_y)=C$ whenever $\mu_1(W_y)\neq 0$. Writing $D\coloneqq\{y\in Y_1\mid \mu_1(W_y)\neq 0\}$ we can thus assume that $W=C\times D\subset X_1\times Y_1$ by composing with $\psi=\dint_{Y_1} \psi_y\,d\nu_1(y)\in [\cR]$.

Passing to the corresponding von Neumann algebras, it follows from (b) above that
\[
u_\theta (p_W L(\cT_0)p_W )u_\theta^* \sub p_Z (L(\cR_2)\otb L^\infty(Y_2))p_Z,
\]
where $p_W$ and $p_Z$ denote the projections in $L^\infty(X)$ onto $W$ and $Z$ respectively. The foregoing reasoning implies that we can assume $p_W=p_C\ot p_D\in L^\infty(X_1)\otb L^\infty(Y_1)=L^\infty(X)$ and hence the previous equation can be written as
\[
u_\theta\left( \dint_{D} p_C L(\cT_{0,y})p_C\,d\nu_1(y)\right) u_\theta^* \sub p_Z (L(\cR_2)\otb L^\infty(Y_2))p_Z.
\]
Moreover, since $\theta(W)=Z$, we have $u_\theta p_W u_\theta^*=p_Z$ and hence moving the $u_\theta$'s to the right, we get
\[
\dint_{D} p_C L(\cT_{0,y})p_C\,d\nu_1(y) \sub (p_C\ot p_D) u_\theta^* (L(\cR_2)\otb L^\infty(Y_2)) u_\theta (p_C\ot p_D).
\]
For notational convenience, we will write $A=L^\infty(X)$, $A_1=L^\infty(X_1)$, $B_1=L^\infty(Y_1)$, $A_2=L^\infty(X_2)$, and $B_2=L^\infty(Y_2)$. Also, we let $P\coloneqq u_\theta^* L(\cR_2)u_\theta$, $A_3\coloneqq u_\theta^* A_2 u_\theta$, and $B_3\coloneqq u_\theta^* B_2 u_\theta$. With this notation, the last inclusion reads
\begin{equation}\label{eq:incl1}
\dint_{D} p_C L(\cT_{0,y})p_C\,d\nu_1(y) \sub (p_C\ot p_D) (P\otb B_3) (p_C\ot p_D).
\end{equation}
Since $u_\theta$ normalizes $A$, we also have
\begin{equation}\label{eq:Cartans}
A = A_1\otb B_1 = A_2\otb B_2 = A_3\otb B_3.
\end{equation}
Now, using the fact that commutants go through direct integrals (see for instance \cite[Proposition~14.1.24]{KR97}), we get by taking relative commutants in \eqref{eq:incl1} that 
\begin{align*}
B_3 (p_C\ot p_D) &\sub \left( \dint_D p_C L(\cT_{0,y})p_C \,d\nu_1(y)\right)'\cap (p_C\ot p_D) L(\cR) (p_C\ot p_D)\\
&= \dint_D \left(p_C L(\cT_{0,y})p_C\right)'\cap (p_C L(\cR_1)\,p_C) \,d\nu_1(y).
\end{align*}
For the last equality, we also used the fact that $\dint_D p_C L(\cT_{0,y})p_C \,d\nu_1(y)$ contains the maximal abelian subalgebra $L^\infty (X)(p_C\ot p_D)$, and thus its relative commutant is contained in $L^\infty(X)(p_C\ot p_D)\subset p_C L(\cR_1)p_C\otb L^\infty(Y_1)p_D = \dint_D (p_C L(\cR_1)\,p_C) \,d\nu_1(y)$. 

Since $\cT_{0,y}\rvert_C$ is ergodic, it follows that $p_C L(\cT_{0,y})p_C$ is a factor. Moreover, this factor contains $A_1 p_C$ which is a Cartan subalgebra, and thus maximal abelian, in the II$_1$ factor $p_C L(\cR_1)p_C$. In particular, it follows that $\left(p_C L(\cT_{0,y})p_C\right)'\cap (p_C L(\cR_1)\,p_C) = \C p_C$. We conclude that
\begin{equation}\label{eq:incl2}
B_3 (p_C\ot p_D) \sub \dint_D \C p_C  \,d\nu_1(y) = p_C\otb B_1p_D.
\end{equation}
From this point, one can apply almost verbatim the end of the proof of \cite[Theorem~G]{IS18}. For the reader's convenience, and later reference, we reproduce the argument here. Thanks to \eqref{eq:incl2}, there exists a von Neumann subalgebra $B_4\subset B_1$ such that $B_3(p_C\ot p_D)=p_C\otimes B_4 p_D$. Taking relative commutants in this equality, we get
\[
p_C L(\cR_1)p_C \,\otb [(B_4 p_D)'\cap p_D L(\cS_1)p_D] = (p_C\ot p_D)(P\otb B_3)(p_C\ot p_D).
\]
In particular, since $P$ is a factor, we see that the center of the above algebra equals $B_3(p_C\ot p_D) = p_C\otimes B_4 p_D$. Identifying this center with $L^\infty(Y)$ for some probability space $(Y,\nu)$ and disintegrating in the above equality we get
\[
\dint_{Y} p_C L(\cR_1)p_C\,\otb N_y\,d\nu(y) = \dint_{Y} p_y P p_y \,d\nu(y),
\]
where we decomposed $N\coloneqq (B_4 p_D)'\cap p_D L(\cS_1)p_D =\dint_{Y} N_y\,d\nu(y)$, and $p_C\ot p_D=\dint_{Y} p_y\,d\nu(y)\in A_3\otb B_3\subset P\otb B_3$. It follows from for instance \cite[Theorem~IV.8.23]{Ta01} that the above identification splits, i.e. for almost every $y\in Y$ we necessarily have
\[
p_C L(\cR_1) p_C\,\otb N_y = p_y P p_y.
\]
Moreover, we have $B_4 p_D\subset B_1p_D\subset (B_4 p_D)'\cap p_D L(\cS_1)p_D$, so we can also decompose $B_1 p_D = \dint_{Y} B_{1,y}\,d\nu(y)\subset \dint_{Y} N_y\,d\nu(y) = N$, where $B_{1,y}\subset N_y$ is a unital inclusion for all $y$. Now $B_1\subset L(\cS_1)$ is a Cartan subalgebra, and hence so is $B_1 p_D\subset p_D L(\cS_1) p_D$. Since $B_1 p_D\subset N\subset p_D L(\cS_1) p_D$, it follows from \cite{Dy63} that also $B_1 p_D\subset N$ is a Cartan subalgebra. From \cite[Lemma~2.2]{Sp17} we then deduce that $B_{1,y}\subset N_y$ is a Cartan subalgebra for almost every $y\in Y$. Furthermore, it follows from \eqref{eq:Cartans} that
\[
\dint_{Y} A_1p_C\,\otb B_{1,y} \,d\nu(y) = \dint_{Y} A_3 p_y \,d\nu(y).
\]
This identification again splits by \cite[Theorem~IV.8.23]{Ta01}, i.e. for almost every $y\in Y$ we have $A_1p_C \,\otb B_{1,y}= A_3 p_y$.

From the above discussion we now deduce the following identification of inclusions of Cartan subalgebras, for almost every $y\in Y$:
\begin{equation}\label{eq:Cartanincl}
(A_1 p_C\,\otb B_{1,y} \subset p_C L(\cR_1) p_C\,\otb N_y) = (A_3 p_y \subset p_y P p_y).
\end{equation}
Writing $\cT_y\coloneqq\cR(B_{1,y}\subset N_y)$, $t\coloneqq\tau(p_C)$, and $s\coloneqq\tau(p_y)$, we thus get for almost every $y\in Y$:
\begin{equation}\label{eq:concl}
\cR_2^{s} \cong \cR(A_2^{s}\subset L(\cR_2)^{s}) \cong \cR(A_3^{s}\subset P^{s})\cong \cR(A_1^t\otb B_{1,y}\subset L(\cR_1)^t\otb N_y) \cong \cR_1^t \times \cT_y,
\end{equation}
where the last isomorphism follows from Lemma~\ref{lem:split}. Taking any $y$ such that the above equations hold and setting $\cT\coloneqq\cT_y$, we get that $\cR_2$ is stably isomorphic to $\cR_1\times\cT$, as desired. 

For the moreover part, assume $\cS_1$ is hyperfinite. Then $L(\cS_1)$ is the hyperfinite II$_1$ factor. Being a subalgebra of $L(\cS_1)$, $N$ is thus amenable by \cite{Co75}. Hence, $N_y$ is an amenable tracial factor for almost every $y$, and is therefore either isomorphic to $\bM_n(\C)$ for some $n\in\N$, or to the hyperfinite II$_1$ factor $R$. Choose again $y$ such that \eqref{eq:concl} holds. Then in the first case we get that $\cT_y$ is finite and hence $\cR_2$ is stably isomorphic to $\cR_1$. In the second case, $\cT_y$ arises as the equivalence relation of a Cartan inclusion $B\subset R$, and as such is a hyperfinite ergodic p.m.p. equivalence relation by \cite{CFW81}. In this case we thus have $\cR_2\cong \cR_1\times \cR_{hyp}$ where $\cR_{hyp}\cong\cS_1$ is a hyperfinite ergodic p.m.p. equivalence relation. This finishes the proof of the Proposition. 
\end{proof}

For the next Lemma, we will need the following stronger version of intertwining, which is closely tied together with strong intertwining of the corresponding von Neumann algebras. As above, let $\cR$ be a countable p.m.p. equivalence relation on $(X,\mu)$. Suppose $\cS\leq \cR$ is a subequivalence relation such that every $\cR$-class contains infinitely many $\cS$-classes. Let $E\sub X$ be a positive measure subset, and $\cT\leq \cR\rvert_E$ a subequivalence relation. 

\begin{definition}
	If for every $\cT$-invariant subset $E\sub X$ of positive measure, there is no sequence $(\theta_n)_n$ in $[\cT\rvert_E]$ such that $\varphi_\cS(\psi\theta_n\psi')\rarrow 0$ for all $\psi,\psi'\in[\cR]$, then we say that $\cT$ \textit{strongly intertwines into} $\cS$, and write $\cT\emb_\cR^s \cS$.
\end{definition}

Note that this condition is in particular implied by strong intertwining of the associated von Neumann algebras, $L(\cT)\emb_{L(\cR)}^s L(\cS)$. We now record a result which appeared in \cite{DHI16} phrased in terms of measure equivalence for groups, but whose proof applies verbatim to get the following result. 

\begin{lemma}[\!\!{\cite[Proposition~3.1]{DHI16}}]\label{lem:twoway}
	Let $\cR_1$, $\cS_1$, $\cR_2$, $\cS_2$ be countable ergodic p.m.p. equivalence relations on standard probability spaces $(X_1,\mu_1)$, $(Y_1,\nu_1)$, $(X_2,\mu_2)$, $(Y_2,\nu_2)$ respectively such that $\cR\coloneqq\cR_1\times\cS_1\cong\cR_2\times\cS_2$. Assume that, as subequivalence relations of $\cR$, we have $\cR_1\emb_\cR\cR_2$ and $\cR_2\emb_\cR^s\cR_1$. 
	
	Then there are positive measure subsets $W,Z\subset X\coloneqq X_1\times Y_1$, a subequivalence relation $\cS_0\leq \cR_2$, and $\phi\in [[\cR]]$, $\phi:Z\rarrow W$, such that
	\begin{enumerate}
		\item $\cS_0\rvert_Z\leq \cR_2\rvert_Z$ has bounded index, and
		\item $(\phi\times\phi)(\cS_0\rvert_Z)\leq \cR_1\rvert_W$ has bounded index.
	\end{enumerate}
	In particular, upon restricting to appropriate subsets, $\cR_1$ and $\cR_2$ admit isomorphic subequivalence relations of bounded index.
\end{lemma}
\begin{proof}
	The proof of \cite[Proposition~3.1]{DHI16} applies verbatim, upon using Lemma~\ref{lem:DHIprodbddindex} from the preliminaries instead of \cite[Lemma~3.3]{DHI16} where it appears in the proof.
\end{proof}

We will use this result to establish the following Proposition for product equivalence relations, which, albeit not equivalent, can be seen as an analogue of the intertwining of relative commutants in the von Neumann algebra setting (cf. \cite[Lemma~3.5]{Va08}). 

\begin{proposition}\label{prop:commutantish}
	Let $\cR_1$, $\cS_1$, $\cR_2$, $\cS_2$ be countable ergodic p.m.p. equivalence relations on standard probability spaces $(X_1,\mu_1)$, $(Y_1,\nu_1)$, $(X_2,\mu_2)$, $(Y_2,\nu_2)$ respectively such that $\cR\coloneqq\cR_1\times\cS_1\cong\cR_2\times\cS_2$. Assume that, as subequivalence relations of $\cR$, we have $\cR_1\emb_\cR\cR_2$ and $\cR_2\emb_\cR\cR_1$.
	
	Then 
	\begin{enumerate}
		\item also $\cS_1\emb_\cR \cS_2$ and $\cS_2\emb_\cR \cS_1$, and
		\item $\cR_1$ is stably isomorphic to $\cR_2$, and $\cS_1$ is stably isomorphic to $\cS_2$
	\end{enumerate}
\end{proposition}
\begin{proof}
	Firstly, we observe that we in fact have that $\cR_2\emb_\cR^s\cR_1$. Indeed, by Corollary~\ref{cor:RvLRprod}, $\cR_2\emb_\cR\cR_1$ is equivalent to the fact that $L(\cR_2)\otb L^\infty(Y_2)\emb_{L(\cR)} L(\cR_1)\otb L^\infty(Y_1)$. Since $L(\cR_2)\otb L^\infty(Y_2)$ is regular inside $L(\cR)$, this implies that $L(\cR_2)\otb L^\infty(Y_2)\emb_{L(\cR)}^s L(\cR_1)\otb L^\infty(Y_1)$ by \cite[Lemma~2.4(3)]{DHI16}, and therefore also $\cR_2\emb_\cR^s\cR_1$.
	
	By Lemma~\ref{lem:twoway}, we can thus find positive measure subsets $W,Z\subset X$, a subequivalence relation $\cS_0\leq \cR_2$, and $\phi\in [[\cR]]$, $\phi:Z\rarrow W$, such that $\cS_0\rvert_Z\leq \cR_2\rvert_Z$ has bounded index $k$, and $\cR_0\coloneqq (\phi\times\phi)(\cS_0\rvert_Z)\leq \cR_1\rvert_W$ has bounded index $l$. As in the beginning of the proof of Proposition~\ref{prop:intertwineprod}, we can write 
	\[
	\cR_0 = \dint_{Y_1} \cR_{0,y}\,d\nu_1(y),
	\]
	and
	\[
	\cS_0 = \dint_{Y_2} \cS_{0,y}\,d\nu_2(y),
	\]
	such that for almost every $y$ the ergodic decomposition of $\cR_{0,y}\rvert_{W_y}$ is atomic with at most $l$ atoms, and the ergodic decomposition of $\cS_{0,y}\rvert_{Z_y}$ is atomic with at most $k$ atoms. By taking further subsets if necessary, we can thus assume that $\cR_{0,y}\rvert_{W_y}$ and $\cS_{0,y}\rvert_{Z_y}$ are ergodic. For $u=u_\phi$, $p=\proj_W$, $q=\proj_Z$, we then get
	\begin{equation}\label{eq:vNaeq}
	pL(\cR_0)p = u\,qL(\cS_0)q\,u^*,
	\end{equation}
	where we can write 
	\[
	pL(\cR_0)p = \dint_{Y_1} p_yL(\cR_{0,y})p_y\,d\nu_1(y),
	\]
	and
	\[
	qL(\cS_0)q = \dint_{Y_2} q_yL(\cS_{0,y})q_y\,d\nu_2(y).
	\]
	Moreover, by the above we have that $p_yL(\cR_{0,y})p_y$ and $q_yL(\cS_{0,y})q_y$ are factors for almost every $y$. Therefore the center of $pL(\cR_0)p$ equals $L^\infty(Y_1)p$ and the center of $qL(\cS_0)q$ equals $L^\infty(Y_2)q$. From \eqref{eq:vNaeq}, it then follows that 
	\begin{equation}\label{eq:LinftyY}
	L^\infty(Y_1)p = u\, L^\infty(Y_2)q\,u^*.
	\end{equation}
	In particular, through conjugating by $u$, we get
	\[
	\cR(L^\infty(Y_1)p\subset pL(\cR)p) \cong \cR(L^\infty(Y_2)q\subset qL(\cR)q).
	\]
	Since $\cR(L^\infty(Y_1)p\subset pL(\cR)p)\cong (\id_{X_1}\times \cS_1)\rvert_W$ and $\cR(L^\infty(Y_2)q\subset qL(\cR)q)\cong (\id_{X_2}\times \cS_2)\rvert_Z$, we get on the von Neumann algebra level a trace preserving automorphism
	\begin{equation}\label{eq:intS}
	\dint_{X_1} p_x L(\cS_1) p_x\,d\mu_1(x) = p(L^\infty(X_1)\otb L(\cS_1))p \cong q(L^\infty(X_2)\otb L(\cS_2))q = \dint_{X_2} q_x L(\cS_2) q_x\,d\mu_2(x),
	\end{equation}
	implemented by $u$, and in such a way that also the canonical Cartan subalgebras are identified. Here we decomposed $p=\dint_{X_1} p_x \,d\mu_1(x)$ and $q=\dint_{X_2} q_x \,d\mu_2(x)$. 
	
	Firstly, this implies in particular that $L^\infty(X_1)\otb L(\cS_1)\emb_{L(\cR)} L^\infty(X_2)\otb L(\cS_2)$ and vice versa. From Corollary \ref{cor:RvLRprod}, we thus deduce that $\cS_1\emb_\cR \cS_2$ and $\cS_2\emb_\cR \cS_1$. 
	
	We proceed by proving from \eqref{eq:intS} that $\cS_1$ is stably isomorphic to $\cS_2$. By reversing the roles of $\cR_i$ and $\cS_i$, the same reasoning will then apply to $\cR_1$ and $\cR_2$.
	
	Applying \cite[Theorem~IV.8.23]{Ta01} (see also \cite[Theorem~1.13]{Sp17}) to \eqref{eq:intS}, we can find null sets $N_1\sub X_1$, $N_2\sub X_2$, a Borel isomorphism $\Phi:X_2\setminus N_2\rarrow X_1\setminus N_1$, with $\Phi(\mu_2)$ equivalent to $\mu_1$, such that the above isomorphism decomposes into tracial isomorphisms $p_x L(\cS_1) p_x\cong q_{\Phi^{-1}(x)} L(\cS_2) q_{\Phi^{-1}(x)}$. Moreover, by the above reasoning, these isomorphisms can be chosen in such a way that also the Cartan subalgebras $L^\infty(Y_1)p_x$ and $L^\infty(Y_2)q_{\Phi^{-1}(x)}$ are identified. Taking some $x\in X_1\setminus N_1$, we thus get
	\[
	\cS_1^t\cong \cR(L^\infty(Y_1)p_x\subset p_x L(\cS_1) p_x)\cong \cR(L^\infty(Y_2)q_{\Phi^{-1}(x)}\subset q_{\Phi^{-1}(x)} L(\cS_2) q_{\Phi^{-1}(x)})\cong \cS_2^s,
	\]
	where we wrote $t=\tau(p_x)$, and $s=\tau(q_{\Phi^{-1}(x)})$. In other words, $\cS_1$ is stably isomorphic to $\cS_2$, finishing the proof of the Proposition.
\end{proof}

\begin{remark}
	On the von Neumann algebra level, this gives the following result in the setting of Proposition~\ref{prop:commutantish}: If $L(\cR_1)\otb L^\infty(Y_1)\emb_{L(\cR)} L(\cR_2)\otb L^\infty(Y_2)$ and $L(\cR_2)\otb L^\infty(Y_2)\emb_{L(\cR)} L(\cR_1)\otb L^\infty(Y_1)$, then, as established in the above proof, $L^\infty(X_1)\otb L(\cS_1)\emb_{L(\cR)} L^\infty(X_2)\otb L(\cS_2)$ and $L^\infty(X_2)\otb L(\cS_2)\emb_{L(\cR)} L^\infty(X_1)\otb L(\cS_1)$. Furthermore, one gets that $\cR_1$ is stably isomorphic to $\cR_2$. Hence $L(\cR_1)$ is stably isomorphic to $L(\cR_2)$, and similarly $L(\cS_1)$ is stably isomorphic to $L(\cS_2)$. 
\end{remark}

For the remainder of this section, we consider diagonal product actions. Using the established results, this will in particular lead to a proof of Theorem~\ref{mthm:SOE}. For convenience we briefly recall the framework established in the Introduction.

\restatenotation*

\begin{proposition}\label{prop:intertwinediag}
	In the setting of Notation~\ref{not:notation}, suppose the actions $\Gamma_1\act (X_1\times Y_1,\mu_1\times\nu_1)$ and $\Gamma_2\act (X_2\times Y_2,\mu_2\times\nu_2)$ generate the same orbit equivalence relation $\cR$ (i.e. the actions are orbit equivalent). Assume that, as subequivalence relations of $\cR$, we have $\cR_1\times \id_{Y_1}\emb_\cR \cR_2\times\id_{Y_2}$.
	
	Then $\cR_1$ is stably isomorphic to a subequivalence relation of $\cR_2$.
	
	Moreover, if additionally $\cR_2\times \id_{Y_2}\emb_\cR \cR_1\times\id_{Y_1}$, then $\cR_1$ is in fact stably isomorphic to $\cR_2$.
\end{proposition}
\begin{remark}\label{rk:assumptions}
	We note that by Lemma~\ref{lem:RvLR}, the condition  $\cR_1\times \id_{Y_1}\emb_\cR \cR_2\times\id_{Y_2}$ is equivalent to $L(\cR_1)\otb L^\infty(Y_1)\emb_{L(\cR)} L(\cR_2)\otb L^\infty(Y_2)$, where $L(\cR_i) = L^\infty(X_i)\cross\Gamma_i^0$. Here we used the fact that, by construction, $\Gamma_i^0$ acts trivially on $Y_i$. Taking relative commutants, we see that this is moreover equivalent to $L^\infty(Y_2)\emb_{L(\cR)} L^\infty(Y_1)$. We note that if this condition would instead be imposed on the $X_i$'s, one can exploit the fact that $\Gamma_i$ acts \textit{freely} on $X_i$. Indeed, in that case $L^\infty(X_i)$ is maximal abelian inside $L^\infty(X_i)\cross\Gamma_i$, and 
	$L^\infty(X_i\times Y_i)$ is maximal abelian inside $L(\cR)$, from which one can---from the intertwining condition on the $L^\infty(X_i)$---deduce that $L^\infty(X_1\times Y_1)$ and $L^\infty(X_2\times Y_2)$ are unitarily conjugate, after which one can reduce the arguments to a purely measure theoretic setting. See for instance \cite[Theorem~4.1]{Io06a}, and \cite[Theorem~4.2]{Dr20v1} for some results in this direction for such diagonal actions. However, $L^\infty(Y_i)$ is not maximal abelian in $L^\infty(Y_i)\cross\Gamma_i$, and so in the proof below we have to take a different approach, similar to the product case considered above.
\end{remark}
\begin{proof}[Proof of Proposition~\ref{prop:intertwinediag}]
The proof follows much the same lines as the proofs of Proposition~\ref{prop:intertwineprod} and Proposition~\ref{prop:commutantish}, so we only indicate the main differences. Firstly, since $\Gamma_1^0$ acts trivially on $Y_1$, we note that as a subequivalence relation of $\cR$, we can indeed write $\cR_1\times\id_{Y_1}$ as a fibered equivalence relation over $Y_1$:
\[
\cR_1\times \id_{Y_1} = \dint_{Y_1} \cR_1\,d\nu_1(y).
\]
A similar statement holds for $\cR_2$. In particular, writing $M=L(\cR)=(L^\infty(X_1)\otb L^\infty(Y_1))\cross\Gamma_1=(L^\infty(X_2)\otb L^\infty(Y_2))\cross\Gamma_2$, with the Cartan subalgebras identified, we have on the von Neumann algebra level that
\[
L^\infty(Y_1)'\cap M = (L^\infty(X_1)\cross\Gamma_1^0)\otb L^\infty(Y_1) = L(\cR_1)\otb L^\infty(Y_1),
\]
and
\[
L^\infty(Y_2)'\cap M = (L^\infty(X_2)\cross\Gamma_2^0)\otb L^\infty(Y_2)=L(\cR_2)\otb L^\infty(Y_2).
\]
Given the aforementioned facts, we can now apply verbatim the first half of the proof of Proposition~\ref{prop:intertwineprod} to find that, up to conjugating by a unitary which normalizes $L^\infty(X_1\times Y_1)$, there exists a projection $p=p_C\ot p_D\in L^\infty(X_1)\otb L^\infty(Y_1)$ such that (cf. \eqref{eq:incl2})
\begin{equation}\label{eq:incY}
	L^\infty(Y_2)(p_C\ot p_D)\subset p_C\ot L^\infty(Y_1) p_D,
\end{equation} 
and thus a von Neumann subalgebra $B_4\subset L^\infty(Y_1)$ such that 
\begin{equation*}
L^\infty(Y_2)(p_C\ot p_D)=p_C\ot B_4 p_D.
\end{equation*}
Taking relative commutants yields
\begin{equation}\label{eq:prime}
(p_C\ot B_4 p_D)'\cap pMp = p(L(\cR_2)\otb L^\infty(Y_2))p.
\end{equation}
Now since $B_4\sub L^\infty(Y_1)$, we see that 
\begin{equation}\label{eq:LR1}
	(L^\infty(Y_1)p)'\cap pL(\cR)p = p_CL(\cR_1)p_C\,\otb\, L^\infty(Y_1)p_D\sub (p_C\ot B_4 p_D)'\cap pMp.	
\end{equation}
Identifying the center of the algebras in \eqref{eq:prime} with $L^\infty(Y)=L^\infty(Y_2)(p_C\ot p_D)=p_C\ot B_4 p_D$ for some probability space $(Y,\nu)$, and disintegrating over this center, we get
\[
\dint_Y N_y\,d\nu(y) = \dint_Y p_y L(\cR_2)p_y\,d\nu(y),
\]
where we wrote $N\coloneqq (p_C\ot B_4 p_D)'\cap pMp = \dint_Y N_y\,d\nu(y)$ and $p_C\ot p_D = \dint_Y p_y\,d\nu(y)$. From~\eqref{eq:LR1}, it follows that $p_CL(\cR_1)p_C\sub N_y$ for almost every $y\in Y$. Moreover, the canonical Cartan subalgebras are still identified (cf. the proof of Proposition~\ref{prop:intertwineprod}). Using once more \cite[Theorem~IV.8.23]{Ta01}, we thus conclude that $\cR_1$ is stably isomorphic to a subequivalence relation of $\cR_2$.

For the moreover part, we can follow a strategy similar to the proof of Proposition~\ref{prop:commutantish}. For this, one observes that Lemma~\ref{lem:twoway} also holds if $\cR_1,\cR_2\leq \cR$ are interpreted in the setting of Notation~\ref{not:notation}, instead of as being parts of a product equivalence relation. To see this, one can run the same proof as \cite[Proposition~3.1]{DHI16}, and observe that \cite[Lemma~3.3]{DHI16} holds for the diagonal actions from Notation~\ref{not:notation} as well. Indeed, one can replace the group $\Gamma_2$ in that proof by a set of coset representatives for $\Gamma_1$, and observe that normality of $\Gamma_1$ is sufficient for the proof. 

Using this, we now run the proof of Proposition~\ref{prop:commutantish} until equation~\eqref{eq:LinftyY} to find a unitary $u\in L(\cR)$ normalizing $L^\infty(X)$ and projections $p,q\in L^\infty(X)$ such that
\[
L^\infty(Y_1)p = u L^\infty(Y_2)q u^*.
\]
Taking relative commutants, we deduce that
\[
\dint_{Y_1} p_yL(\cR_1)p_y\,d\nu_1(y) = p(L(\cR_1)\otb L^\infty(Y_1))p \cong q(L(\cR_2)\otb L^\infty(Y_2))q = \dint_{Y_2} q_yL(\cR_2)q_y\,d\nu_2(y),
\]
where we decomposed $p = \dint_{Y_1} p_y\,d\nu_1(y)$ and $q = \dint_{Y_2} q_y\,d\nu_2(y)$. The final part of the proof of Proposition~\ref{prop:commutantish} now applies to get the desired result.
\end{proof}

\begin{proposition}\label{prop:int}
	Assume the setting of Notation~\ref{not:notation}. Furthermore assume that $\Sigma_2$ is amenable, $\Gamma_1\act (X_1,\mu_1)$ is strongly ergodic, and the actions $\Gamma_1\act (X_1\times Y_1,\mu_1\times\nu_1)$ and $\Gamma_2\act (X_2\times Y_2,\mu_2\times\nu_2)$ generate the same orbit equivalence relation $\cR$. Then $\cR_1\times \id_{Y_1}\emb_\cR \cR_2\times\id_{Y_2}$.
\end{proposition}

\begin{proof}
Write $M\coloneqq L(\cR)$ and $X\coloneqq X_1\times Y_1 = X_2\times Y_2$. As observed in Remark~\ref{rk:assumptions}, it suffices to show that $L^\infty(Y_2)\emb_M L^\infty(Y_1)$. Note that we can assume without loss of generality that $Y_2$ is non-atomic. To show the intertwining condition, we first observe that since $\Sigma_2$ is amenable and $\Sigma_2\act (Y_2,\nu_2)$ is free, ergodic, and p.m.p., the associated orbit equivalence relation is a hyperfinite ergodic p.m.p. equivalence relation. Thus, we can identify $(Y_2,\nu_2)=(Y_0,\nu_0)^{\mathbb N}$, where $Y_0=\{0,1\}$ and $\nu_0=\frac{1}{2}(\delta_0+\delta_1)$, in such a way that $(y_k)_k$ and $(z_k)_k$ belong to the same orbit if and only if $y_k=z_k$ for all but finitely many $k\in\mathbb N$. We can then identify $L^\infty(Y_2)=\otb_{k\in\mathbb N}L^{\infty}(Y_0,\nu_0)_k$, and we put $B_{n}=\otb_{k\geq n}L^{\infty}(Y_0,\nu_0)_k$, for $n\in\mathbb N$. Fixing an ultrafilter $\omega$ on $\N$, it is then not hard to show that
\[
\prod_\omega B_{n}\subset M'\cap L^\infty(X)^\omega,
\]
see for instance the first half of the proof of \cite[Lemma~6.1]{IS18}. Moreover, by strong ergodicity of the action of $\Gamma_1$ on $X_1$, we also have
\[
M'\cap L^\infty(X)^\omega\subset L^\infty(Y_1)^\omega. 
\]
Combining both inclusions, we thus get $\prod_\omega B_{n}\subset L^\infty(Y_1)^\omega$. Applying \cite[Lemma~2.3]{IS18} then gives a sequence $\eps_n\rarrow 0$ such that $B_{n}\subset_{\eps_n} L^\infty(Y_1)$. In particular $B_{n}\emb_M L^\infty(Y_1)$ for large enough $n$. Since $B_{n}$ has finite index in $L^\infty(Y_2)$, this also gives $L^\infty(Y_2) \emb_M L^\infty(Y_1)$ and thus finishes the proof of the Proposition.
\end{proof}

\begin{proof}[Proof of Theorem~\ref{mthm:SOE}]
From Proposition~\ref{prop:int} applied twice, we get that $\cR_1\times\id_{Y_1}\emb_\cR \cR_2\times\id_{Y_2}$ and $\cR_2\times\id_{Y_2}\emb_\cR \cR_1\times\id_{Y_1}$. The moreover assertion of Proposition~\ref{prop:intertwinediag} then implies the result. 
\end{proof}

\section{(Non-)Schmidt equivalence relations}\label{sec:Schmidt}

In this section we consider (non-)Schmidt equivalence relations and prove Theorem~\ref{mthm:Schmidt}. First of all, identifying elements of $[[\cR]]$ with partial isometries in $L(\cR)$, and denoting by $\norm{.}_2$ the 2-norm on $L(\cR)$ corresponding to the canonical trace, we can give a ``local'' characterization of the Schmidt property, similar to the local characterization of stability in \cite[Theorem~2.1]{Ma17}.

\begin{lemma}\label{lem:local}
	Let $\cR$ be an ergodic p.m.p. equivalence relation on a standard probability space $(X,\mu)$. Then the following are equivalent:
	\begin{enumerate}[label=(\arabic*)]
		\item $\cR$ is Schmidt.
		\item For every finite set $F\subset [[\cR]]$ and every $\eps>0$, there exists $v\in [[\cR]]$ such that
		\[
		\forall u\in F:\quad \norm{vu-uv}_2 < \eps \norm{v-\id\rvert_{\dom(v)}}_2.
		\]
	\end{enumerate}
\end{lemma}
\begin{proof}
$(1)\Rightarrow (2)$ is immediate. Indeed, by definition, we can find $\delta>0$ and a central sequence $(T_n)_n$ in $[\cR]$ such that $\mu(\{x\in X\mid T_n(x)\neq x\})\geq \delta$ for all $n$. Moreover, we observe that for any $v\in [[\cR]]$, we have $\norm{v-\id\rvert_{\dom(v)}}_2^2 = 2\mu(\{x\in \dom(v)\mid v(x)\neq x\})$. Therefore, given any finite set $F\subset [[\cR]]$, we can take $n$ large enough so that
\[
\forall u\in F:\quad \norm{T_nu-uT_n}_2 < \eps \sqrt{2\delta} \leq \eps \norm{T_n-\id}_2,
\]
and thus (2) holds. 

Before proving $(2)\Rightarrow (1)$, we first show that if $\cR$ satisfies (2), then for every subset $Y\subset X$ of positive measure, $\cR\rvert_Y$ satisfies (2) as well. The argument is, mutatis mutandis, essentially the same as the first part of the proof of (iii)$\Rightarrow$(ii) in \cite[Theorem~2.1]{Ma17}, but we include it for the reader's convenience. 

Let $Y\subset X$ be a subset of positive measure and set $p=1_Y$. Suppose that $\cR\rvert_{Y}$ does not satisfy (2). Then we can find a finite set $F\subset [[\cR\rvert_Y]]$ and a constant $\kappa>0$ such that for all $v\in [[\cR\rvert_Y]]$, we have
\begin{equation}\label{eq:(2)}
	\norm{v-\id\rvert_{\dom(v)}}_2^2\leq \kappa\sum_{u\in F} \norm{vu-uv}_2^2.
\end{equation}
Since $\cR$ is assumed ergodic, we can find a finite set $K\subset [[\cR]]$ such that $\sum_{w\in K} w^*w = p^\perp$ and $ww^*\leq p$ for all $w\in K$. Now take any $v\in [[\cR]]$. Then we can write
\[
\norm{v-\id\rvert_{\dom(v)}}_2^2 = \norm{p(v-\id\rvert_{\dom(v)})}_2^2 + \sum_{w\in K} \norm{w(v-\id\rvert_{\dom(v)})}_2^2.
\]
At this point, we also observe that we can write $\id\rvert_{\dom(v)} = v^*v$. Thus, for any $w\in K$, we have 
\begin{align*}
	\norm{w(v-\id\rvert_{\dom(v)})}_2^2 &= \norm{wv-wv^*v}_2^2\\ 
	&\leq 4(\norm{wv - vw}_2^2 + \norm{vw - v^*vw}_2^2 + \norm{v^*vw - v^*wv}_2^2 + \norm{v^*wv - wv^*v}_2^2)\\
	&\leq 4(\norm{wv - vw}_2^2 + \norm{vp - v^*vp}_2^2 + \norm{vw - wv}_2^2 + \norm{v^*w - wv^*}_2^2).
\end{align*}
Combining the above, we thus have
\begin{align*}
	\norm{v-\id\rvert_{\dom(v)}}_2^2 \leq 4\abs{K}(\norm{p(v-v^*v)}_2^2 + &\norm{(v - v^*v)p}_2^2)\\ &+ 8\sum_{w\in K} \norm{vw-wv}_2^2 + 4\sum_{w\in K^*} \norm{vw-wv}_2^2.
\end{align*}
Furthermore, we see that $v^*vp = pv^*v = pv^*vp = \id_{\dom(v)\cap Y} = \id_{\dom(pvp)}$, from which a direct calculation gives
\begin{align*}
	\norm{p(v-v^*v)}_2^2 + \norm{(v - v^*v)p}_2^2 &= \norm{p(v-v^*v) - (v - v^*v)p}_2^2 + 2\norm{p(v-v^*v)p}_2^2\\
	&= \norm{pv - vp}_2^2 + 2\norm{pvp-\id_{\dom(pvp)}}_2^2.
\end{align*}
Applying \eqref{eq:(2)} to $pvp\in [[\cR\rvert_{Y}]]$, we get
\[
\norm{pvp-\id_{\dom(pvp)}}_2^2 \leq \kappa\sum_{u\in F} \norm{(pvp)u-u(pvp)}_2^2 \leq \kappa\sum_{u\in F} \norm{vu-uv}_2^2.
\]
Combining everything, we thus have
\begin{align*}
	\norm{v-\id\rvert_{\dom(v)}}_2^2 \leq 4\abs{K}\norm{pv - vp}_2^2 + \,&8\abs{K}\kappa \sum_{u\in F} \norm{vu-uv}_2^2 \\
	&+ 8\sum_{w\in K} \norm{vw-wv}_2^2 + 4\sum_{w\in K^*} \norm{vw-wv}_2^2.
\end{align*}
This shows $\cR$ does not satisfy (2), contradiction. Hence for every subset $Y\subset X$ of positive measure, $\cR\rvert_Y$ satisfies (2) as well.

We now prove $(2)\Rightarrow (1)$. Take a $\norm{.}_2$-dense sequence $(S_n)_n$ in $[\cR]$. Using (2), we can find for every $n\in\N$ an element $v\in [[\cR]]$ such that 
\[
\forall i\leq n: \quad \norm{vS_i - S_iv}_2 < \frac{1}{n}\norm{v-\id\rvert_{\dom(v)}}_2.
\]
Let $v_n$ be such an element for which $\mu(\{x\in X\mid v_n(x)\neq x\})$ is maximal among all such elements. Noting that $v_n^*v_n$ is the identity map on $A_n\coloneqq \dom(v_n)$ and $v_nv_n^*$ is the identity map on $B_n\coloneqq \im(v_n)$, we see that $(A_n)_n$ and $(B_n)_n$ are asymptotically invariant sequences in $X$, and therefore so are $(A_n\cap B_n)_n$ and $(A_n\setminus (A_n\cap B_n))_n$.

\textbf{Case 1.} $\lim\limits_{n\rarrow\infty}\mu(A_n\setminus (A_n\cap B_n)) \neq 0$. 

Passing to a subsequence if necessary, we can in this case assume that there exists $\delta>0$ such that for all $n$, $\mu(A_n\setminus (A_n\cap B_n))\geq \delta$. Restricting $v_n$ to $A_n\setminus (A_n\cap B_n)$, we then find a sequence witnessing condition (iii) in \cite[Theorem~2.1]{Ma17}. Therefore, $\cR$ is stable and hence Schmidt.

\textbf{Case 2.} $\lim\limits_{n\rarrow\infty}\mu(A_n\setminus (A_n\cap B_n)) = 0$. 

We claim that in this case, $\mu(\{x\in \dom(v_n)\mid v_n(x)\neq x\})\rarrow 1$ (and thus in particular also $\mu(\dom(v_n))\rarrow 1$). By extending $v_n$ in any way on $X\setminus \dom(v_n)$ to an element $T_n\in [\cR]$, this then shows that $\cR$ is Schmidt.

We first of all show that for $v_n$ chosen as above, $\lim\limits_{n\rarrow\infty} \mu(A_n) = 1$. Indeed, if not, then taking a subsequence if necessary, there exists $\delta>0$ such that $\mu(A_n)\leq 1-\delta$ for all $n$. By the assumption of Case~2, we thus get $\mu(X\setminus (A_n\cup B_n))\geq \delta/2$ for large enough $n$. Fix such $n$. Write $Y\coloneqq X\setminus (A_n\cup B_n)$, and apply (2) to $\cR\rvert_{Y}$. This gives $w_n\in [[\cR\rvert_Y]]$ such that 
\[
\forall i\leq n:\quad \norm{pS_ipw_n-w_npS_ip}_2 < \frac{1}{n} \norm{w_n-\id\rvert_{\dom(w_n)}}_2,
\]
where $p=1_{Y}$ denotes the projection onto $Y$. Extending $v_n$ to equal $w_n$ on $\dom(w_n)\subset Y$ then contradicts the maximality of $v_n$.

Similarly, if $\lim\limits_{n\rarrow\infty}\mu(\{x\in \dom(v_n)\mid v_n(x)\neq x\})\neq 1$, we consider the set $Y_n\coloneqq \{x\in \dom(v_n)\mid v_n(x)= x\}$. After passing to a subsequence if necessary, we can then assume $\mu(Y_n)\geq\delta$ for all $n$ and some $\delta>0$. Applying (2) to $Y_n$, produces $w_n\in [[\cR\rvert_{Y_n}]]$ such that 
\[
\forall i\leq n:\quad \norm{pS_ipw_n-w_npS_ip}_2 < \frac{1}{n} \norm{w_n-\id\rvert_{\dom(w_n)}}_2,
\]
Define the element $\tilde v_n\in [[\cR]]$ to equal $v_n$ on $X\setminus Y_n$ (which is obviously $v_n$-invariant) and to equal $w_n$ on $Y_n$. This element then once again contradicts the maximality of $v_n$, finishing the proof of Case~2, and therefore the Lemma.
\end{proof}

\begin{remark}\label{remark:measureone}
	Note that the proof in particular shows that if $\cR$ is Schmidt, then there exists a central sequence $T_n\in [\cR]$ such that $\mu(\{x\in X\mid T_n(x)\neq x\})=1$ for all $n$. This fact was first observed in \cite[Lemma~5.5]{KTD18}.
\end{remark}

Next, we will use Lemma~\ref{lem:local} to provide a characterization of central sequences in the full group of a product equivalence relation, one of whose factors is not Schmidt. Before stating and proving this result, we establish some notation which will be useful in the proof. We also refer to \cite[Section~3]{Ma17} where a similar notation is introduced and exploited.

Let $\cR$ be a p.m.p. equivalence relation on a standard probability space $(X,\mu)$, and identify $L^2(\cR,\bar \mu)$ with the canonical $L^2$-space of the von Neumann algebra $L(\cR)$. For any $x\in L(\cR)$, we will write $\hat x$ for the corresponding vector in $L^2(\cR,\bar \mu)$. Most important for us is that for $v\in [[\cR]]$, $\hat v$ corresponds to the indicator function of the graph $\{(v(x),x)\mid x\in \dom(v)\}$ of $v$.

Now if $\cS$ is another p.m.p. equivalence relation on some standard probability space $(Y,\nu)$, then for any $v\in [[\cR\times\cS]]$, there exists a unique function $v_\cR\in L^0(\cS,[[\cR]])$ such that 
\[
\hat v(x,x',y,y') = \widehat{v_\cR(y,y')}(x,x')
\]
for almost every $(x,x',y,y')\in \cR\times\cS$. In other words, $v_\cR$ satisfies
\begin{align*}
v_\cR(y,y')(x) = x' \;&\iff\; \widehat{v_\cR(y,y')}(x,x') = 1\\
&\iff\; \hat v(x,x',y,y') = 1\\
&\iff\; v(x,y) = (x',y')
\end{align*}
for almost every $(x,x',y,y')\in \cR\times\cS$.

\begin{lemma}\label{lem:cseqSchmidt}
	Let $\cR$ and $\cS$ be countable ergodic p.m.p. equivalence relations on standard probability spaces $(X,\mu)$ and $(Y,\nu)$ respectively, and assume $\cR$ is not Schmidt. If $(v_n)_n$ is a central sequence in $[[\cR\times\cS]]$, then there exists a measurable assignment $X\ni x\mapsto v_{n,x}\in [[\cS]]$ such that $(v_{n,x})_n$ is a central sequence for almost every $x\in X$, and such that
	\[
	(\mu\times\nu)(\{(x,y)\in X\times Y\mid v_n(x,y)\neq (x,v_{n,x}(y))\})\rarrow 0\quad\text{ as $n\rarrow\infty$}.
	\]
	Moreover, if $v_n\in [\cR\times\cS]$ for every $n$, then $v_{n,x}\in [\cS]$ for every $n$ and almost every $x\in X$.
\end{lemma}
\begin{proof}	
Let $(v_n)_n\in [[\cR\times \cS]]$ be a central sequence. Write $v_n(x,y) = (v_n^{(1)}(x,y),v_n^{(2)}(x,y))$, and consider $v_{n,\cR}\in L^0(\cS,[[\cR]])$. The Lemma will follow if we prove that
\[
(\mu\times\nu)(\{(x,y)\in X\times Y\mid v_n^{(1)}(x,y)\neq x\})\rarrow 0.
\]
Since $\cR$ is not Schmidt, it follows from Lemma~\ref{lem:local} that there exists a finite set $F\subset [[\cR]]$ and $\kappa>0$ such that
\[
\forall v\in [[\cR]]:\quad \norm{v-\id\rvert_{\dom(v)}}_2^2 \leq \kappa \sum_{u\in F} \norm{vu-uv}_2^2.
\]
Moreover, we recall the observation that $\norm{v-\id\rvert_{\dom(v)}}_2^2 = 2\mu(\{x\in \dom(v)\mid v(x)\neq x\})$. We now compute:
\begin{align*}
	(\mu\times\nu)(\{(x,y)\in \;&X\times Y\mid v_n^{(1)}(x,y)\neq x\}) \\
	&= (\mu\times\nu)(\{(x,y)\in X\times Y\mid v_{n,\cR}(y,v_n^{(2)}(x,y))(x)\neq x\})\\
	&= (\mu\times\bar\nu)(\{(x,y,y')\in X\times \cS\mid v_{n,\cR}(y,y')(x)\neq x\})\\
	&= \int_{\cS} \mu(\{x\in X\mid v_{n,\cR}(y,y')(x)\neq x\}) \,d\bar{\nu}(y,y')\\
	&\leq \kappa/2 \sum_{u\in F} \int_{\cS} \norm{v_{n,\cR}(y,y')u - uv_{n,\cR}(y,y')}_2^2 \,d\bar{\nu}(y,y')\\
	&= \kappa/2 \sum_{u\in F} \int_{\cS} \int_\cR \big|\hat{v}_n(x,x',y,y')\hat{u}(x,x') -\\
		&\qquad\qquad\qquad\qquad\qquad \hat{u}(x,x')\hat{v}_n(x,x',y,y')\big|^2 \,d\bar{\mu}(x,x')\,d\bar{\nu}(y,y')\\
	&= \kappa/2 \sum_{u\in F} \norm{v_n (u\ot 1) - (u\ot 1) v_n}_2^2.
\end{align*}
Since $(v_n)_n$ is by assumption a central sequence in $[[\cR\times\cS]]$, the latter converges to zero, finishing the proof of the main part of the Lemma. The moreover part follows immediately.
\end{proof}

The next Lemma tells us that we can recognize the absence of the Schmidt property in the stabilization of an equivalence relation.

\begin{lemma}\label{lem:nonSchmidt}
	Let $\cR_1, \cR_2$ be countable ergodic p.m.p. equivalence relations on $(X_1,\mu_1)$ and $(X_2,\mu_2)$ respectively, and $\cR_{0,1}, \cR_{0,2}$ be hyperfinite ergodic p.m.p. equivalence relations on $(Y_1,\nu_1)$ and $(Y_2,\nu_2)$ respectively, such that $\cR_1\times\cR_{0,1}\cong \cR_2\times\cR_{0,2}$. Assume that $\cR_1$ is not Schmidt. Then either
	\begin{enumerate}
		\item $\cR_2$ is also not Schmidt, or
		\item $\cR_2$ is stable, and thus $\cR_2\cong \cR_2\times\cR_{0,2}\cong \cR_1\times\cR_{0,1}$.
	\end{enumerate}
\end{lemma}
\begin{proof}
	For notational convenience, we will identify $(X_1\times Y_1,\mu_1\times \nu_1)$ and $(X_2\times Y_2,\mu_2\times \nu_2)$ through the given isomorphism, and write $\mu=\mu_1\times\nu_1 = \mu_2\times\nu_2$. Assume by contradiction that $\cR_2$ is Schmidt and not stable. Following Remark~\ref{remark:measureone}, since $\cR_2$ is Schmidt, we can find a central sequence $(T_n)_n$ in $[\cR_2]$ such that for all $n$,
	\begin{equation*}
		\mu_2(\{x\in X_2\mid T_n(x)\neq x\})=1.
	\end{equation*}
	Furthermore, since $\cR_2$ is not stable, we must have by \cite[Theorem~3.4]{JS85} that for every asymptotically invariant sequence $(A_n)_n\sub X_2$ for $\cR_2$, 
	\begin{equation*}
		\lim\limits_{n\rarrow\infty} \mu_2(T_nA_n\Delta A_n)\rarrow 0.
	\end{equation*}
	By considering $T_n\coloneqq T_n\times\id_{Y_2}\in [\cR_2\times\cR_{0,2}]$, we then get from the above that also
	\begin{equation}\label{eq:ac1}
		\mu(\{(x,y)\in X_2\times Y_2\mid T_n(x,y)\neq (x,y)\}) = (\mu_2\times\nu_2)(\{(x,y)\in X_2\times Y_2\mid T_n(x)\neq x\}) =1,
	\end{equation}
	and for any asymptotically invariant sequence $(A_n)_n\sub X_2\times Y_2$, 
	\begin{equation}\label{eq:ac2}
		\lim\limits_{n\rarrow\infty} \mu(T_nA_n\Delta A_n)\rarrow 0.
	\end{equation}
	For the latter, note that if $(A_n)_n$ is an asymptotically invariant sequence in $X_2\times Y_2$, it is in particular asymptotically invariant under all maps of the form $S\times\id_{Y_2}\in [\cR_2\times \cR_{0,2}]$. Therefore, writing $A_{n,y}\coloneqq \{x\in X_2\mid (x,y)\in A_n\}$, the sets $(A_{n,y})_n$ have to be asymptotically invariant for $\cR_2$ for almost every $y\in Y_2$. Thus, necessarily $\lim_{n\rarrow\infty} \mu_2(T_nA_{n,y}\Delta A_{n,y})\rarrow 0$ for almost every $y\in Y_2$. Integrating over $Y_2$, \eqref{eq:ac2} then follows from the dominated convergence theorem.
	
	Through the given isomorphism, we now view $T_n\in [\cR_1\times\cR_{0,1}]$. Since $\cR_1$ is not Schmidt, Lemma~\ref{lem:cseqSchmidt} provides a measurable assignment $x\mapsto T_{n,x}\in [\cR_{0,1}]$ such that $(T_{n,x})_n$ is central for almost every $x\in X_1$ and such that 
	\begin{equation}\label{eq:Tnx}
		\mu(\{(x,y)\in X_1\times Y_1\mid T_n(x,y)\neq (x,T_{n,x}(y))\})\rarrow 0.
	\end{equation}
	Combining \eqref{eq:Tnx} and \eqref{eq:ac1}, we get that
	\begin{equation}\label{eq:Tny}
		\nu_1(\{y\in Y_1\mid T_{n,x}(y)\neq y\})\rarrow 1
	\end{equation}
	for almost every $x\in X_1$.
	
	Since $\cR_{0,1}$ is ergodic, hyperfinite, and p.m.p., we can realize it as the orbit equivalence relation arising from the action of $\oplus_\N \Z/2\Z$ on $Y_1=\prod_\N \Z/2\Z$ by componentwise translation. 
	Passing to a subsequence if necessary, and using the fact that $(T_{n,x})_n$ is a central sequence for the hyperfinite equivalence relation $\cR_{0,1}$ for almost every $x\in X_1$, we claim that we can find $V_{n,x}\in [\cR_{0,1}]$ of the form 
	\begin{equation}\label{eq:Vnx}
	V_{n,x} = \id_{\prod_{i\leq n}\Z/2\Z} \times \tilde V_{n,x}
	\end{equation}
	such that
	\begin{equation}\label{eq:TV}
		\nu_1(\{y\in Y_1\mid T_{n,x}(y)\neq V_{n,x}(y)\})\rarrow 0
	\end{equation}
	for almost every $x\in X_1$.
	Indeed, for a fixed $m\geq 2$, we can consider the (finitely many) elements $\{S_k\}$ of the full group $[\cR_{0,1}]$ corresponding to bijections $\{0,1\}^m\rarrow \{0,1\}^m$ which only change the first $m$ coordinates of a given element in $\prod_\N \Z/2\Z$. Note that these elements permute the clopen subsets $\{C_a\mid a\in \{0,1\}^m\}$ where $C_a \coloneqq \{y=(y_i)_i\in Y_1\mid y_i=a_i \text{ for all } i\leq m\}$. Since the $(T_{n,x})_n$ are central sequences, we can take $n$ large enough so that 
	\begin{equation}\label{eq:TSST}
		\nu_1(\{y\in Y_1\mid T_{n,x}S_k(y) = S_kT_{n,x}(y) \text{ for all } k\})>1-\frac{1}{m}
	\end{equation}
	for all $x$ in a subset $X^{(m)}$ of $X_1$ of measure at least $1-\frac{1}{m}$.
	Now suppose $y\in Y_1$ is such that $T_{n,x}(y)_i\neq y_i$ for some $i\leq m$ and some $x\in X^{(m)}$, and let $a\neq b\in \{0,1\}^m$ be such that $y\in C_a$ and $T_{n,x}(y)\in C_b$. Choose any element $S_k$ as above that leaves $C_a$ invariant and maps $C_b$ onto $C_c$ for some $c\neq a,b$. Then $T_{n,x}(S_k(y)) = T_{n,x}(y)\in C_b$, but $S_k(T_{n,x}(y))\in C_c$, and therefore
	\[
	T_{n,x}S_k(y) \neq S_kT_{n,x}(y).
	\]
	Combining this with \eqref{eq:TSST}, we conclude that for every $x\in X^{(m)}$, the set $Z_x\coloneqq \{y\in Y_1\mid T_{n,x}(y)_i\neq y_i \text{ for some } i\leq m\}$ has measure
	\[
	\nu_1(Z_x)\leq \frac{1}{m}.
	\]
	For $x\in X^{(m)}$, let $V_{m,x}$ equal $T_{n,x}$ on $Y_1\setminus Z_x$ and extend $V_{m,x}$ in any measurable way to $Z_x$ such that it leaves the first $m$ coordinates of any element fixed. For $x\notin X^{(m)}$, we can define $V_{m,x}$ in any way that leaves the first $m$ coordinates of any element fixed. Then the $V_{m,x}$ will satisfy \eqref{eq:Vnx}. 
	
	Note that since $x\mapsto T_{n,x}$ is measurable, the same holds for $x\mapsto Z_x$ and thus for $x\mapsto V_{m,x}|_{Y_1\setminus Z_x}$ on $X^{(m)}$. Since we are free in extending $V_{m,x}$ on $Z_x$ and defining $V_{m,x}$ for $x\notin X^{(m)}$, we can make sure $x\mapsto V_{m,x}$ is measurable on $X_1$.
	
	Next, observe that by construction, the map
	\[
	V_n \coloneqq \dint_{X_1} V_{n,x}\,dx: X_1\times Y_1\rarrow X_1\times Y_1: (x,y)\mapsto (x,V_{n,x}(y))
	\]
	is asymptotically invariant for $\cR_1\times\cR_{0,1}$. In particular, we get for any $S\in [\cR_1]$ that
	\begin{align*}
	\mu(\{(x,y)&\in X_1\times Y_1\mid (S\times\id)V_n(x,y)\neq V_n(S\times\id)(x,y)\}) \\
	&= \mu(\{(x,y)\in X_1\times Y_1\mid (S(x),V_{n,x}(y))\neq (S(x),V_{n,S(x)}(y))\}) \\
	&\longrightarrow 0.
	\end{align*}
	Hence for every $S\in [\cR_1]$, we have that for almost every $x\in X_1$:
	\[
	\nu_1(\{y\in Y_1\mid V_{n,x}(y)\neq V_{n,S(x)}(y)\}\rarrow 0.
	\]
	Fix a countable dense sequence $S_k\in [\cR_1]$ and $\frac{1}{32}>\eps_n>0$ with $\lim_{n\rarrow\infty} \eps_n=0$. Passing to a subsequence of $(V_n)_n$ if necessary, we can then find subsets $\tilde X_{(n)}\subset X_1$ of measure at least $1-\eps_n$ such that for all $x\in \tilde X_{(n)}$ and all $1\leq k\leq n$:
	\begin{equation}\label{eq:xS}
	\nu_1(\{y\in Y_1\mid V_{n,x}(y)\neq V_{n,S_k(x)}(y)\}<\eps_n.
	\end{equation}
	
	Combining \eqref{eq:TV} with \eqref{eq:Tny}, we note that we also have 
	\[
	\nu_1(\{y\in Y_1\mid V_{n,x}(y)\neq y\})\rarrow 1,
	\]
	for almost every $x\in X_1$. Passing to a further subsequence if necessary, we can thus assume that the set $X'_{(n)}\coloneqq \{x\in \tilde X_{(n)}\mid \nu_1(\{y\in Y_1\mid V_{n,x}(y)\neq y\})\geq 1-\eps_n\}$ has measure $\mu_1(X'_{(n)})\geq 1-2\eps_n$.

	{\bf Claim.} For every $n\geq 1$, we can find a measurable map $X_1\ni x\mapsto A_{n,x}\subset Y_1$ and a set $X_{(n)}$ with $\mu_1(X_{(n)})\geq 1-3\eps_n$, such that the following four conditions hold.
	\begin{enumerate}
		\item For all $x\in X_{(n)}$: $V_{n,x}(A_{n,x})\cap A_{n,x} =\emptyset$,
		\item for all $x\in X_{(n)}$ and $1\leq k\leq n$: $\nu_1(A_{n,x}\Delta A_{n,S_k(x)})<2\eps_n$, 
		\item $\int_{X_{(n)}} \nu_1(A_{n,x}) \,d\mu_1(x) \geq\frac{1}{4}$, and
		\item $A_{n,x}$ is of the form $A_{n,x} = \prod_{i\leq n} \Z/2\Z \times \tilde A_{n,x}$.
	\end{enumerate}
	
	{\it Proof of the Claim.} Fix a Borel metric $d$ with diameter 1 on $Y_1$, for instance through a Borel isomorphism $\theta$ with $[0,1]$. For $x\in X'_{(n)}$, define
	\[
	\delta_x\coloneqq \sup\{\delta>0\mid \nu_1(\{y\in Y_1\mid d(y,V_{n,x}(y))>\delta\})>1-2\eps_n\}.
	\]
	Then $x\mapsto \delta_x$ is measurable, and we can find $\delta_0>0$ such that 
	\[
	\mu_1(\{x\in X'_{(n)}\mid \delta_x\geq\delta_0\})>1-3\eps_n.
	\]
	Denote the set of such $x$ by $X_{(n)}$. Fix $N\geq\frac{1}{\delta_0}$, and partition $Y_1$ into $N$ Borel sets $B_1, \ldots, B_N$ of diameter $\leq\frac{1}{N}$, for instance by pushing forward the intervals $[\frac{i-1}{N}, \frac{i}{N}]$ through the Borel isomorphism $\theta$. By construction, given $x\in X_{(n)}$, we have that the measure of the set of $y\in Y_1$ for which $y$ and $V_{n,x}(y)$ lie in different $B_i$'s is at least $1-2\eps_n$, and thus
	\[
	\nu_1(\{y\in Y_1\mid y\text{ lies in a different }B_i\text{ than both }V_{n,x}(y)\text{ and }V_{n,x}^{-1}(y)\})>1-4\eps_n.
	\]	
	Denote the set of such $y$ by $Y_x$. Consider an ordering of the sets $B_1, \ldots, B_N$, and define an associated partial order on $Y_1$ by 
	\[
	y_1 < y_2 \,\Longleftrightarrow\, y_1\in B_i, y_2\in B_j \text{ and } B_i < B_j.
	\]
	For $x\in X_{(n)}$, we then consider the set
	\[
	A_{n,x,<}\coloneqq \{y\in Y_x\mid y < V_{n,x}(y) \text{ and } y < V_{n,x}^{-1}(y)\}.
	\]
	Firstly, we observe that for every ordering $<$, the map $x\mapsto A_{n,x,<}$ is measurable. Furthermore, for every ordering $<$, the sets $A_{n,x,<}$ satisfy property (1) from the claim by construction, and property (2) follows from \eqref{eq:xS}. For property (3), we observe that for every $x\in X_{(n)}$ and $y\in Y_x$, we have that $y$, $V_{n,x}(y)$, and $V_{n,x}^{-1}(y)$ lie in two or three different $B_i$'s (depending on whether $V_{n,x}(y)$ and $V_{n,x}^{-1}(y)$ lie in the same one or not). In particular, it follows that for every $x\in X_{(n)}$, each $y\in Y_x$ lies in $A_{n,x,<}$ for at least $1/3$rd of the orderings of $B_1, \ldots, B_N$. Summing over all possible $N!$ orderings, we get
	\[
	\sum_{i=1}^{N!} \nu_1(A_{n,x,<_i}) \geq \frac{N!}{3}\cdot(1-4\eps_n)
	\]
	for every $x\in X_{(n)}$, and thus
	\[
	\sum_{i=1}^{N!} \int_{X_{(n)}} \nu_1(A_{n,x,<_i}) \,d\mu_1(x) = \int_{X_{(n)}} \sum_{i=1}^{N!} \nu_1(A_{n,x,<_i}) \,d\mu_1(x) \geq \frac{N!}{3}\cdot(1-4\eps_n)\cdot (1-3\eps_n) \geq \frac{N!}{4}.
	\]
	In particular there exists an ordering $<$ of the sets $B_1, \ldots, B_N$ such that the sets $A_{n,x}\coloneqq A_{n,x,<}$ satisfy $\int_{X_{(n)}} \nu_1(A_{n,x}) \,d\mu_1(x) \geq\frac{1}{4}$. Finally, property (4) follows from running the above argument on $\prod_{i>n} \Z/2\Z$, which we can do thanks to \eqref{eq:Vnx}, finishing the proof of the claim. \hfill(End Claim $\square$)

	We now define
	\[
	A_n\coloneqq \bigcup_{x\in X_{(n)}} \{x\}\times A_{n,x} \subset X_1\times Y_1.
	\]
	(One can also view this as a direct integral.) By the above Claim(3), we have $\mu(A_n)\geq \frac{1}{4}>0$ for every $n$. Moreover, by construction $(A_n)_n$ is an asymptotically invariant sequence for $\cR_1\times \cR_{0,1}\cong \cR_2\times\cR_{0,2}$. Indeed, it is asymptotically invariant for $\cR_1$ by Claim(2) together with the fact that $\mu_1(X_{(n)})\rarrow 1$, and asymptotically invariant for $\cR_{0,1}$ thanks to Claim(4).
		
	Furthermore, by Claim(1) and the foregoing paragraph, we have
	\[
	\liminf_{n\rarrow\infty} \mu(V_nA_n\Delta A_n)> 0.
	\]
	Combining this with \eqref{eq:TV} and \eqref{eq:Tnx}, we thus get
	\[
	\liminf_{n\rarrow\infty} \mu(T_nA_n\Delta A_n)> 0.
	\]
	However, this is a contradiction with \eqref{eq:ac2}, and thus finishes the proof of the Lemma.
\end{proof}

The last Lemma we will need in order to complete the proof of Theorem~\ref{mthm:Schmidt} deduces intertwining phenomena in the presence of a non-Schmidt equivalence relation.

\begin{lemma}\label{lem:Schmidt}
	Let $\cR_1, \cR_2$ be countable ergodic p.m.p. equivalence relations on $(X_1,\mu_1)$ and $(X_2,\mu_2)$ respectively, and suppose $\cR_{0,1}, \cR_{0,2}$ are hyperfinite ergodic p.m.p. equivalence relations on $(Y_1,\nu_1)$ and $(Y_2,\nu_2)$ respectively. Assume that $\cR_1$ is not Schmidt, and that $\cR\coloneqq\cR_1\times \cR_{0,1} \cong \cR_2\times \cR_{0,2}$. Then $\cR_{0,2}\emb_\cR \cR_{0,1}$.
\end{lemma}
\begin{proof}
	For notational convenience, we will identify $(X_1\times Y_1,\mu_1\times \nu_1)$ and $(X_2\times Y_2,\mu_2\times \nu_2)$ through the given isomorphism, and denote it by $(Z,\mu)$. Assume that $\cR_{0,2}\not\emb_\cR \cR_{0,1}$. 	
	
	{\bf Claim.} We can find a sequence $(\theta_n)_{n=1}^\infty\in [\id_{X_2}\!\times\cR_{0,2}]$ which is central for $\cR$, and such that $\varphi_{\id_{X_1}\!\!\times\cR_{0,1}}(\theta_n)\rarrow 0$.
	
	{\it Proof of the Claim.}	
	Since $\cR_{0,2}$ is ergodic, hyperfinite, and p.m.p., we can again realize it as the orbit equivalence relation arising from the action of $\oplus_\N \Z/2\Z$ on $\prod_\N \Z/2\Z$ by componentwise translation. For a fixed $n\in\N$, we denote by $\cR_{0,2}^{(n)}$ the equivalence relation $\cR(\oplus_{m>n} \Z/2\Z \act \prod_\N \Z/2\Z)$. By construction, $\cR_{0,2}^{(n)}\leq \cR_{0,2}$ has finite index. As subequivalence relations of $\cR$, since $\cR_{0,2}\not\emb_\cR \cR_{0,1}$, we thus also get that $\cR_{0,2}^{(n)}\not\emb_\cR \cR_{0,1}$. Passing to the von Neumann algebra level, we get $L^\infty(X_2)\otb L(\cR_{0,2}^{(n)})\not\emb_{L(\cR)} L^\infty(X_1)\otb L(\cR_{0,1})$. Therefore, by Theorem~\ref{thm:Popaint}, we get a sequence of unitaries $u_k$ in any subgroup $\cU\subset \cU(L^\infty(X_2)\otb L(\cR_{0,2}^{(n)}))$ that generates $L^\infty(X_2)\otb L(\cR_{0,2}^{(n)})$, satisfying $\norm{E_{L^\infty(X_1)\otb L(\cR_{0,1})}(u_k)}_2\rarrow 0$. In particular, this holds for unitaries of the form $v\ot u_\theta$, with $v\in L^\infty(X_2)$ and $\theta\in [\cR_{0,2}^{(n)}]$, which we identify with a subset of $[\cR_{0,2}]$ by identifying $\theta$ with $\id_{\prod_{i\leq n}\Z/2\Z}\times\theta$. Hence we can find unitaries $v_k\in L^\infty(X_2)$ and $\theta_k^{(n)}\in [\cR_{0,2}^{(n)}]$ such that
	\begin{equation}\label{eq:disappear}
	\norm{E_{L^\infty(X_1)\otb L(\cR_{0,1})}(v_k\ot u_{\theta_k^{(n)}})}_2 = \norm{E_{L^\infty(X_1)\otb L(\cR_{0,1})}(1\ot u_{\theta_k^{(n)}})}_2 \rarrow 0,
	\end{equation}
	where the equality follows from the fact that $v_k\in L^\infty(X_2)\subset L^\infty(X_1)\otb L(\cR_{0,1})$ and the fact that the conditional expectation is bimodular. A straightforward calculation moreover implies that
	\[
	\norm{E_{L^\infty(X_1)\otb L(\cR_{0,1})}(1\ot u_{\theta_k^{(n)}})}_2^2 = \varphi_{\id_{X_1}\!\!\times\cR_{0,1}}(\theta_k^{(n)}),
	\]
	hence the latter converges to zero as well. Choosing $k_n\geq n$ such that $\varphi_{\id_{X_1}\!\!\times\cR_{0,1}}(\theta_{k_n}^{(n)})\leq \frac{1}{n}$ and putting $\theta_n\coloneqq\theta_{k_n}^{(n)}$ we now have a sequence $(\theta_n)_n$ in $[\id_{X_2}\!\times\cR_{0,2}]$ such that
	\begin{enumerate}[label=(\alph*)]
		\item $\varphi_{\id_{X_1}\!\!\times\cR_{0,1}}(\theta_n)\rarrow 0$ as $n\rarrow\infty$, and
		\item $\theta_n$ is of the form $\theta_n = \id_{X_2}\times\, (\id_{\prod_{i\leq n}\Z/2\Z} \,\times \,\tilde\theta_n)$.
	\end{enumerate}
	Since it clearly follows from (b) that $(\theta_n)_n$ forms a central sequence in $[\cR_2\times\cR_{0,2}]=[\cR]$, this finishes the proof of Claim.
	\hfill(End Claim $\square$) 
	
	Taking $(\theta_n)_n$ as in the Claim, applying Lemma~\ref{lem:cseqSchmidt} together with the fact that $\cR_1$ is not Schmidt gives central sequences $(\theta_{n,x})_n$ in $[[\cR_{0,1}]]$ for almost every $x\in X_1$ such that 
	\[
	\mu(\{(x,y)\in X_1\times Y_1\mid \theta_n(x,y)\neq (x,\theta_{n,x}(y))\})\rarrow 0.
	\]
	In particular, denoting by $\theta_n^{(1)}(x,y)$ the first coordinate of $\theta_n(x,y)\in X_1\times Y_1$, we get that
	\begin{equation}\label{eq:tozero}
	\mu(\{(x,y)\in X_1\times Y_1\mid \theta_n^{(1)}(x,y)\neq x\})\rarrow 0.
	\end{equation}
	However, the fact that $\varphi_{\id_{X_1}\!\!\times\cR_{0,1}}(\theta_n)\rarrow 0$ translates to
	\begin{align*}
	\mu(\{(x,y)\in X_1\times Y_1\mid \;&\theta_n^{(1)}(x,y)=x\}) \\
	&= \mu(\{(x,y)\in X_1\times Y_1\mid \theta_n(x,y)(\id_{X_1}\!\times\cR_{0,1})(x,y)\})\\
	&= \varphi_{\id_{X_1}\!\!\times\cR_{0,1}}(\theta_n) \rarrow 0.
	\end{align*}
	This contradicts \eqref{eq:tozero}, and we conclude that indeed $\cR_{0,2}\emb_\cR \cR_{0,1}$.
\end{proof}

\begin{remark}
	From the proof of the Claim in the above Lemma, we notice that the following stronger statement holds for intertwining of equivalence relations, similar to the von Neumann algebra case: In the setting of Lemma~\ref{lem:inteqrel}, let $G\subset [\cT]$ be a subgroup that generates $\cT$. Then the following are equivalent.
	\begin{enumerate}
		\item $\cT\emb_\cR \cS$.
		\item There is no sequence $\{\theta_n\}_{n=1}^\infty\sub G$ such that $\varphi_\cS(\psi\theta_n\psi')\rarrow 0$ for all $\psi,\psi'\in [\cR]$.
	\end{enumerate}
	The reason is that the involved von Neumann algebras $L(\cT)$, $L(\cS)$, and $L(\cR)$ all contain the Cartan subalgebra $L^\infty(X)$, and so unitaries from $L^\infty(X)$ will disappear in the calculation as in \eqref{eq:disappear} above. Then one uses \cite[Lemma~1.8]{Io11} (Lemma~\ref{lem:RvLR} above) and the observation that unitaries from $L^\infty(X)$ together with unitaries of the form $u_\theta$, $\theta\in [\cT]$, generate $L(\cT)$.
\end{remark}

\begin{proof}[Proof of Theorem~\ref{mthm:Schmidt}]
In the setting of Theorem~\ref{mthm:Schmidt}, write $\cR\coloneqq \cR_1\times\cR_{0,1}\cong \cR_2\times\cR_{0,2}$, where $\cR_{0,1}$ and $\cR_{0,2}$ denote two copies of $\cR_{hyp}$. From Lemma~\ref{lem:nonSchmidt} we get that $\cR_2$ is either not Schmidt or stable. In the latter case, item (2) of Theorem~\ref{mthm:Schmidt} holds. In the former case, we get from applying  Lemma~\ref{lem:Schmidt} twice that $\cR_{0,1}\emb_\cR\cR_{0,2}$ and $\cR_{0,2}\emb_\cR\cR_{0,1}$. An application of Proposition~\ref{prop:commutantish} then finishes the proof.
\end{proof}

\end{document}